\documentclass[11pt]{amsart}

\usepackage{amsmath, amsthm, amssymb, mathtools, mathrsfs, stmaryrd}
\usepackage{ascmac}
\usepackage{comment}
\usepackage{bm}
\allowdisplaybreaks
\usepackage{hyperref}

\usepackage{graphicx}
\usepackage{here}
\usepackage{time}
\usepackage[abbrev]{amsrefs}

\usepackage{xcolor}
\usepackage[capitalize,nameinlink,noabbrev,nosort]{cleveref}
\hypersetup{
	colorlinks=true,       
	linkcolor=black,          
	citecolor=black,        
	filecolor=black,      
	urlcolor=black,           
}

\usepackage[textheight=8.9in, textwidth=6.9in]{geometry}

\makeatletter
\@namedef{subjclassname@2020}{\textup{2020} Mathematics Subject Classification}
\makeatother

\newtheorem*{theorem*}{Theorem}

\newtheorem{theoremcounter}{Theorem Counter}[section]

\theoremstyle{remark}
\newtheorem*{remark*}{Remark}

\theoremstyle{definition}

\newtheorem*{example*}{Example}

\theoremstyle{plain}
\newtheorem{lemma}[theoremcounter]{Lemma}
\newtheorem{proposition}[theoremcounter]{Proposition}
\newtheorem{corollary}[theoremcounter]{Corollary}

\newtheorem{theorem}[theoremcounter]{Theorem}

\numberwithin{equation}{section}

\newcommand{\Z}{\mathbb{Z}}

\newcommand{\R}{\mathbb{R}}
\newcommand{\C}{\mathbb{C}}

\newcommand{\bt}{\mathbf{t}}
\newcommand{\e}{\mathbf{e}}
\newcommand{\fe}{\mathfrak{e}}
\newcommand{\dd}{\mathrm{d}}

\newcommand{\bbH}{\mathbb{H}}

\newcommand{\calG}{\mathcal{G}}
\newcommand{\calM}{\mathcal{M}}
\newcommand{\calO}{\mathcal{O}}

\newcommand{\calW}{\mathcal{W}}

\DeclareMathOperator{\ImNew}{Im}
\renewcommand{\Im}{\ImNew}
\DeclareMathOperator{\ReNew}{Re}
\renewcommand{\Re}{\ReNew}

\DeclareMathOperator{\Mp}{Mp}
\DeclareMathOperator{\SL}{SL}
\DeclareMathOperator{\PSL}{PSL}
\DeclareMathOperator{\GL}{GL}
\DeclareMathOperator{\sgn}{sgn}

\newcommand{\pmat}[1]{\begin{pmatrix}#1\end{pmatrix}}

\newcommand{\smat}[1]{\bigl(\begin{smallmatrix}#1\end{smallmatrix}\bigr)}

\def\ov#1{\overline{#1}}
\def\un#1{\underline{#1}}

\begin{document}

\title[]{Eichler--Selberg relations for singular moduli} 

\author{Yuqi Deng}
\address{Graduate School of Mathematics, Kyushu University, Motooka 744, Nishi-ku, Fukuoka 819-0395, Japan}
\email{deng.yuqi.608@s.kyushu-u.ac.jp}

\author{Toshiki Matsusaka}
\address{Faculty of Mathematics, Kyushu University, Motooka 744, Nishi-ku, Fukuoka 819-0395, Japan}
\email{matsusaka@math.kyushu-u.ac.jp}

\author{Ken Ono}
\address{Department of Mathematics, University of Virginia, Charlottesville, VA 22904, USA}
\email{ko5wk@virginia.edu}




\maketitle

\begin{abstract}  The Eichler--Selberg trace formula expresses the trace of Hecke operators  on spaces of cusp forms as weighted sums of Hurwitz--Kronecker class numbers. We extend this formula to a natural class of relations for traces of singular moduli, where one views class numbers as traces of the constant function $j_0(\tau)=1$. More generally, we consider the singular moduli for the Hecke system of modular functions
$$
j_m(\tau) \coloneqq  mT_m\left( j(\tau)-744\right).
$$ 
For each $\nu\geq 0$ and  $m\geq 1$, we obtain an {\it Eichler--Selberg relation}. For $\nu=0$ and $m\in \{1, 2\},$ these relations are
Kaneko's celebrated singular moduli formulas for the coefficients of $j(\tau).$ For each $\nu\geq 1$ and  $m\geq 1,$ we obtain a new Eichler--Selberg trace formula for the Hecke action on the space of weight $2\nu+2$ cusp forms, where the traces of $j_m(\tau)$ singular moduli replace Hurwitz--Kronecker class numbers. These formulas involve a new term that is assembled from values of symmetrized shifted convolution $L$-functions.
\end{abstract}


\section{Introduction and Statement of Results}\label{section-intro}

Let $j(\tau)$ be the usual modular function for $\SL_2(\Z)$ with Fourier expansion
$$
j(\tau)=q^{-1}+744+196884q+21493760q^2+\cdots,
$$
where $q \coloneqq e^{2\pi i \tau}.$ Its values at imaginary quadratic arguments in the upper-half of the complex plane are  examples of {\it singular moduli} \cite{Zagier2002}. They are algebraic integers that generate Hilbert class fields of imaginary quadratic fields, in addition to serving as isomorphism class invariants of elliptic curves with complex multiplication.
Well-known examples of these values include:
$$
j\left(\frac{1+\sqrt{-3}}{2}\right)=0,\ \ j(i)=1728, \ \ {\text {\rm and}}\ \ 
j\left(\frac{1+\sqrt{-15}}{2}\right)=\frac{-191025-85995\sqrt{5}}{2}.
$$

We consider the sequence of modular functions $j_0(\tau) \coloneqq 1, j_1(\tau) \coloneqq j(\tau)-744,\dots$ that satisfy
$$
j_m(\tau) = q^{-m} + O(q).
$$
Each $j_m(\tau)$ is a monic degree $m$ polynomial in $\Z[j(\tau)],$ and the set
 $\{j_m(\tau) : m \ge 0\}$ is a basis of $M^!_0$,  the space of weakly holomorphic modular functions on $\SL_2(\Z)$.
The first examples are $j_0(\tau) = 1$ and
	\begin{align*}
		j_1(\tau) &= j(\tau) - 744 = q^{-1} + 196884q + \cdots,\\
		j_2(\tau) &= j(\tau)^2 - 1488j(\tau) + 159768 = q^{-2} + 42987520q + \cdots,\\
		j_3(\tau) &= j(\tau)^3 - 2232 j(\tau)^2 + 1069956 j(\tau) - 36866976 = q^{-3} + 2592899910q + \cdots.
	\end{align*}
In terms of the  Hecke operators $T_m$ (see ~\cite[Ch.~VII]{Serre1973} and \cite{Zagier2002}), for positive integers $m$ we have
\begin{equation}\label{jm}
j_m(\tau)=q^{-m}+\sum_{n=1}^{\infty}c_m(n)q^n= mT_m\left(j(\tau)-744\right).
\end{equation}

We shall derive infinitely many relations for the singular moduli of these functions.
To make this precise,
for positive integers $d$ with $-d\equiv 0,1 \pmod 4$, we let $\mathcal{Q}_d$ be the set of integral positive definite binary quadratic forms $Q(X,Y) = [A,B,C] \coloneqq AX^2 + BXY + CY^2$ with discriminant $-d=B^2-4AC$. The group $\Gamma \coloneqq \PSL_2(\Z)$ acts on $\mathcal{Q}_d$ by
\[
	\left(Q \circ \pmat{a & b \\ c & d}\right) (X,Y)  \coloneqq  Q(aX+bY, cX+dY),
\]
and does so with finitely many orbits, the number of which is the discriminant $-d$ {\it class number}.
For each $Q \in \mathcal{Q}_d$, we let $\alpha_Q \in \bbH$ be a root of $Q(\tau,1) = 0.$ The numbers $j_m(\alpha_Q)$ are its {\it singular moduli}.

We study the weighted traces of these values which are defined as follows.
If we let $\Gamma_Q$ be the stabilizer of $Q$ in $\Gamma$, then it is well known that
\[
	\#\Gamma_Q = \begin{cases}
		3 &\text{if } Q \sim a(X^2 + XY + Y^2),\\
		2 &\text{if } Q \sim a(X^2 + Y^2),\\
		1 &\text{if otherwise}.
	\end{cases}
\]
Following Zagier \cite{Zagier2002}, the {\it trace functions} we consider are
\begin{equation}
	\bt_m(d) \coloneqq \sum_{Q \in \mathcal{Q}_d/\Gamma} \frac{j_m(\alpha_Q)}{\# \Gamma_Q}.
\end{equation}
For $m = 0$, where $j_0(\tau)=1,$ we obtain the {\it Hurwitz--Kronecker} class numbers
$H(d) \coloneqq \bt_0(d).$  These  numbers are prominent in the Eichler--Selberg trace formula (for example, see \cite{Zagier1991}) for the action of the Hecke operators on $S_{2k},$ the complex vector space of  weight $2k$ cusp forms on $\SL_2(\Z).$ 
\begin{theorem*}[The Eichler--Selberg trace formula]\label{Thm:ES}
	For integers $k \ge 2$, we have
	\begin{equation}\label{ESTF}
		\mathrm{Tr}(n; 2k) = -\frac{1}{2} \sum_{r \in \Z} p_{2k} (r,n) \bt_0(4n-r^2) - \lambda_{2k-1}(n),
	\end{equation}
	where $\lambda_k(n) \coloneqq \frac{1}{2} \sum_{d \mid n} \min(d, n/d)^k$ and 
	$$p_k(r,n) = \sum_{0 \le j \le \frac{k}{2}-1} (-1)^j \binom{k-2-j}{j} n^j r^{k-2-2j} = \mathrm{Coeff}_{X^{k-2}} \left(\frac{1}{1-rX+nX^2}\right).
	$$
\end{theorem*}

We generalize these formulas to traces of singular moduli, where \eqref{ESTF} are the $m=0$ cases of a doubly infinite suite of formulas in $m\geq 0$ and  $\nu\geq 0.$ The general formulas involve the trace functions $\bt_m(4n-r^2)$. 
To make this  precise, for every $\nu\geq 0$ and $m\geq 0,$ we define the generating function
\begin{equation}\label{KeyGF}
	\calG_{m,\nu}(\tau) \coloneqq -\frac{1}{2}\sum_{n \gg -\infty} \sum_{r \in \Z} p_{2\nu+2}(r,n) \bt_m(4n-r^2) q^n,
\end{equation}
where for $d\le 0$ we let
\begin{equation}\label{eq:tm-neg}
	\bt_m(d) \coloneqq \begin{cases}
		2\sigma_1(m) &\text{if } d=0,\\
		-\kappa &\text{if } d = -\kappa^2 \text{ and } \kappa \mid m,\\
		0 &\text{if otherwise}.
	\end{cases}
\end{equation}
By \eqref{ESTF}, each $\mathrm{Tr}(n;2k)$ is essentially the $n$th coefficient of $\calG_{0,k-1}(\tau).$ 
Therefore, we refer to any explicit formula for $\calG_{m,\nu}(\tau)$ as an Eichler--Selberg relation for $m$ and $\nu.$

Our first result establishes that these generating functions are {\it weakly holomorphic modular forms},
meromorphic modular forms whose poles (if any) are supported at cusps.
For convenience, we let $M^!_k$ denote the space of such weight $k$ forms on $\SL_2(\Z).$ 

\begin{theorem}\label{FundGF}
If $\nu\geq 0$ and $m\geq 1,$ then we have that
$\calG_{m,\nu}(\tau)\in M^!_{2\nu+2}.$
\end{theorem}

The $\nu=0$ Eichler--Selberg relations only involve derivatives of the $j_m(\tau)$, as they generate $M^!_2$  due to the absence of holomorphic modular forms.  For convenience, we let $D \coloneqq \frac{1}{2\pi i}\frac{\dd}{\dd \tau} = q \frac{\dd}{\dd q}$.

\begin{theorem}\label{EasyEichlerSelberg_nu0} For positive integers $m,$ the following are true.

\noindent
1) We have 
\[
	\calG_{m,0}(\tau)	= -\frac{1}{2}\sum_{\kappa \mid m} \sum_{0 < r < \kappa} \frac{\kappa}{r(\kappa-r)} \cdot Dj_{r(\kappa-r)}(\tau).
\]

\noindent
2) If $n$ is a positive integer, then we have
\[
	\sum_{r \in \Z} \bt_m(4n-r^2) = n \sum_{\kappa \mid m} \sum_{0 < r < \kappa} \frac{\kappa}{r(\kappa-r)} c_{r(\kappa-r)}(n).
\]
\end{theorem}

\begin{example*} 
	\cref{EasyEichlerSelberg_nu0}, with $m\in \{1, 2\}$, gives Kaneko's identities~\cite{Kaneko1996}
	$$
	\sum_{r\in \Z} \bt_1(4n-r^2)=0\ \ \ \ {\text {\rm and}}\ \ \ \
	\sum_{r\in \Z} \bt_2(4n-r^2)=2nc_1(n),
	$$
	which he used to derive his well-known singular moduli formula for the coefficients of $j(\tau)$
	$$
	c_1(n)=\frac{1}{n}\left \{ \sum_{r \in \Z} \bt_1(n-r^2)+\sum_{r\geq 1\ odd}\left( (-1)^n \bt_1(4n-r^2)-\bt_1(16n-r^2)\right)\right\}.
	$$
Such formulas have been extended to higher levels $N$ in subsequent works~\cite{Ohta2009, MatsusakaOsanai2017, Matsusaka2017}. 
\end{example*}

For $\nu > 0$, there are holomorphic modular forms, and so the relations have richer structure. To make this precise, we recall the weight $2k$ modular Poincar\'e series~\cite[Ch.~6.3]{BFOR2017}
\begin{align}\label{def:classical-Poincare}
	P_{2k,h}(\tau) \coloneqq \sum_{\gamma \in \Gamma_\infty \backslash \Gamma} q^h |_{2k} \gamma,
\end{align}
where $|_{2k}$ is the slash operator, $\Gamma=\PSL_2(\Z)$, and $\Gamma_{\infty}$ is the stabilizer for the cusp infinity.
The usual Eisenstein series is $P_{2k,0}(\tau) = E_{2k}(\tau),$ and for negative integers $-h$, we have the weakly holomorphic
\begin{equation}\label{PP}
P_{2k,-h}(\tau)=q^{-h}+\sum_{n=1}^{\infty}c_{2k,-h}(n)q^n.
\end{equation}

For small $\nu,$ when there are no cusp forms, we obtain the following Eichler--Selberg relations.

\begin{theorem}\label{EasyEichlerSelberg} 
	If $\nu \in \{1, 2, 3, 4, 6\},$  then for every  positive integer $m$ the following are true.
	
	\noindent
	1) We have that
	\[
		\calG_{m,\nu}(\tau) = \sum_{\kappa \mid m} \sum_{0 < r \le \kappa} r^{2\nu+1} P_{2\nu+2, -r(\kappa-r)}(\tau).
	\]
	
	\noindent
	2) If $n$ is a positive integer, then we have
	\[
		\sum_{r \in \Z} p_{2\nu+2}(r,n) \bt_m(4n-r^2) = -2 \sum_{\kappa \mid m} \sum_{0 < r \le \kappa} r^{2\nu+1} c_{2\nu+2, -r(\kappa-r)}(n).
	\]
\end{theorem}

\begin{remark*} The Poincar\'e series in \cref{EasyEichlerSelberg} are easily described in terms of the Eisenstein series 
$$
E_4(\tau)=1+240\sum_{n=1}^{\infty}\sigma_3(n)q^n \ \ \ {\text {\rm and}}\ \ \
E_6(\tau)=1-504\sum_{n=1}^{\infty} \sigma_5(n)q^n.
$$
For $k\in \{4, 6, 8, 10, 14\},$ we have
$$
 P_{k,-1}(\tau) = \begin{cases}
E_4(\tau)\cdot (j(\tau)-984) &{\text {\rm if $k=4$,}}\\
E_6(\tau)\cdot (j(\tau)-240) &{\text {\rm if $k=6$,}}\\
E_4^2(\tau)\cdot (j(\tau)-1224) &{\text {\rm if $k=8$,}}\\
E_4(\tau)E_6(\tau)\cdot (j(\tau)-480) &{\text {\rm if $k=10$,}}\\
E_4^2(\tau)E_6(\tau)\cdot (j(\tau)-720) &{\text {\rm if $k=14$.}}
\end{cases}
$$
Generalizing \eqref{jm}, for $m>1,$ we have the Hecke formula
$$
P_{k,-m}(\tau) = m^{-k+1} \cdot T_m P_{k,-1}(\tau).
$$
\end{remark*}

\begin{example*} For positive integers $n,$ \cref{EasyEichlerSelberg} with $\nu = 1$ and $m=1$ implies that
$$
\sum_{r\in \Z} r^2 \bt_1(4n-r^2)=-480\sigma_3(n).
$$
\end{example*}

Cusp forms arise in the general case. Special values of symmetrized shifted convolution $L$-functions and Petersson norms control these cusp forms in these Eichler--Selberg relations. Throughout, we let $d_{2k}$ denote the dimension of $S_{2k}$, the space of weight $2k$ cusp forms on $\SL_2(\Z).$

\begin{theorem}\label{Main-Theorem}
	If $\nu \geq 1$ and $m \ge 1$,
		then we have
	\begin{align*}
		\calG_{m,\nu}(\tau)=
		 &\sum_{\kappa \mid m} \sum_{0 < r \le \kappa} r^{2\nu+1} P_{2\nu+2, -r(\kappa-r)}(\tau)
		 -\sum_{j=1}^{d_{2\nu+2}} \left(24\sigma_1(m) - \frac{\Gamma(2\nu+1)}{(4\pi)^{2\nu+1}} \frac{\widehat{L}(f_j, m; 2\nu+1)}{\|f_j\|^2} \right) f_j,
	\end{align*}
	where the $f_j$'s are normalized Hecke eigenforms of $S_{2\nu+2}$ and 
	\[
		\widehat{L}(f, m; s) \coloneqq \sum_{n=1}^\infty \frac{c_f(n) c_f(n+m)}{n^s} - \sum_{n=1}^\infty \frac{c_f(n) c_f(n-m)}{n^s}.
	\]
\end{theorem}

\begin{example*} Example 1 of ~\cite{MertensOno2016} gives $\widehat{L}(\Delta, 1; 11)=-33.383\dots$ and $\widehat{L}(\Delta, 2; 11)=266.439\dots,$ which arise in \cref{Main-Theorem} when $\nu = 5$ and $m\in \{1, 2\}.$  By brute force computation, we have
\begin{align*}
		\calG_{1,5}(\tau) &=  E_{12}(\tau) -\frac{82104}{691} \Delta(\tau),\\
		\calG_{2,5}(\tau) &= P_{12,-1}(\tau) + 2049 E_{12}(\tau) - \left(\alpha - \frac{1746612}{691} \right) \Delta(\tau),
 	\end{align*}
	where
	\begin{align*}
		P_{12,-1}(\tau) &= \Delta(\tau) (j_2(\tau) + 24 j_1(\tau) + 324 + \alpha) = q^{-1} + \alpha q + \cdots,
	\end{align*}
	with $\alpha = 1842.894\ldots$.
Using $\|\Delta\|^2 = \langle \Delta, \Delta \rangle = 0.0000010353\ldots,$ these numerics illustrate \cref{Main-Theorem}
	\begin{align*}
		\frac{82104}{691} &= 24 + \frac{65520}{691} = 24 - \frac{\Gamma(11)}{(4\pi)^{11}} \frac{(-33.383\ldots)}{\|\Delta\|^2},\\
		\alpha - \frac{1746612}{691} &= 24 \cdot 3 - \frac{\Gamma(11)}{(4\pi)^{11}} \frac{(266.439\ldots)}{\|\Delta\|^2}.
	\end{align*}
	\end{example*}

\cref{Main-Theorem} gives a doubly infinite family of modified Eichler--Selberg trace formulas, where Hecke eigenvalues are weighted by shifted convolution $L$-values, and where traces of singular moduli $\bt_m(4n-r^2)$ replace the Hurwitz--Kronecker class numbers $\bt_0(4n-r^2)=H(4n-r^2).$  
To make this precise, we let
\begin{equation}\label{ModifiedTrace}
\mathrm{Tr}_m(n;2k) \coloneqq \frac{\Gamma(2k-1)}{(4\pi)^{2k-1}} \sum_{j=1}^{d_{2k}} \frac{\widehat{L}(f_j, m; 2k-1)}{\|f_j\|^2} \cdot c_{f_j}(n),
\end{equation}
where, as above,  $c_{f_j}(n)$ is the eigenvalue of $T_n$ for the Hecke eigenform $f_j\in S_{2k}.$

\begin{corollary}\label{GeneralizedES} If $2k \in 2\Z^+\setminus\{2, 4, 6, 8, 10, 14\}$ and $m$ is a positive integer, then we have
$$
 \mathrm{Tr}(n; 2k) =\frac{1}{ 24 \sigma_1(m)}\cdot\left( \mathrm{Tr}_m(n; 2k) +\frac{1}{2} \sum_{r \in \Z} p_{2k}(r,n) \bt_m(4n-r^2) + \sum_{\kappa \mid m} \sum_{0 < r \le \kappa} r^{2k-1} c_{2k, -r(\kappa-r)}(n)\right).
$$
\end{corollary}

To obtain these results, we adapt Zagier's novel (unpublished) proof \cite{Zagier1991} of the Eichler--Selberg trace formula. In \cref{NovelZagier} we recall his proof and his work on traces of singular moduli, and we prove Theorems~\ref{FundGF}-\ref{EasyEichlerSelberg}.
The proof of \cref{Main-Theorem} is more involved, as we make use of the theory of vector-valued Poincar\'e series, the arithmetic of half-integral weight Kloosterman sums, Rankin--Cohen bracket operators,  and symmetrized shifted convolution $L$-functions. In \cref{Section3} we recall important formalities regarding vector-valued modular forms that transform according to the Weil representation.  In \cref{Section4} we relate the Fourier coefficients of half-integral weight Maass--Poincar\'e series to traces of singular moduli, and finally in \cref{Section5} we assemble these facts to prove \cref{Main-Theorem}.

\section*{Acknowledgments}
This work was supported by the MEXT Initiative for Realizing Diversity in the Research Environment through Kyushu University's Diversity and Super Global Training Program for Female and Young Faculty (SENTAN-Q). The second author was supported by JSPS KAKENHI (JP20K14292, JP21K18141, and JP24K16901). The third author thanks the Thomas Jefferson Fund and grants from the NSF
	(DMS-2002265 and DMS-2055118).

\section{Zagier's work and the proofs of Theorems~\ref{FundGF}-\ref{EasyEichlerSelberg}}\label{NovelZagier}

In unpublished notes~\cite{Zagier1991}, Zagier gave a novel proof of the Eichler--Selberg trace formula using harmonic Maass forms  (see \cite{BruinierFunke2004} or \cite{BFOR2017} for background on harmonic Maass forms).  Saad and the third author ~\cite{OnoSaad2022} obtained further such formulas by modifying his argument.
We adapt his argument in a different aspect.

\subsection{Zagier's Proof}
We begin by sketching his proof, which relies on the following theorem.

\begin{theorem*}[Zagier \cite{Zagier1975}] We have that
$$
		\mathcal{H}(\tau) \coloneqq -\frac{1}{12} + \sum_{\substack{d > 0 \\ d \equiv 0, 3\ (\rm{mod}\ 4)}} H(d) q^d + \frac{1}{8\pi \sqrt{v}} + \frac{1}{4\sqrt{\pi}} \sum_{n=1}^\infty n \Gamma \left(-\frac{1}{2}; 4\pi n^2 v\right) q^{-n^2}
$$
	is a harmonic Maass form of weight $3/2$ on $\Gamma_0(4)$, where $\tau = u + iv$ and $\Gamma(s; x)$ is the incomplete Gamma function.
Its {\it holomorphic part} is the Fourier series
\[
	\mathcal{H}^+(\tau) \coloneqq -\frac{1}{12} + \sum_{\substack{d > 0 \\ d \equiv 0, 3\ (\rm{mod}\ 4)}} H(d) q^d.
\]
\end{theorem*}

Zagier uses a sequence of modular forms he constructs from $\mathcal{H}(\tau)$ and Jacobi's weight 1/2 theta function 
\begin{equation}\label{Jacobi}
\theta(\tau) \coloneqq \sum_{n \in \Z} q^{n^2}=1+2q+2q^4+\cdots.
\end{equation}
To define these modular forms, he requires Atkin's $U$-operator defined by
\begin{equation}\label{U}
		(f| U_m)(\tau) \coloneqq \frac{1}{m} \sum_{j=0}^{m-1} f \left(\frac{\tau+j}{m}\right),
\end{equation}
and the Rankin--Cohen bracket differential operators.
For modular forms $f$ and $g$ (possibly non-holomorphic), with weights $k$ and $l$ respectively, these operators are defined by
	\begin{equation}\label{RC}
		[f,g]_\nu \coloneqq \sum_{\substack{r,s \ge 0 \\ r+s = \nu}} (-1)^r \frac{\Gamma(k+\nu) \Gamma(l+\nu)}{s! \Gamma(k+r) r! \Gamma(l+s)} D^r(f) D^s(g),
	\end{equation}
	where $D = \frac{1}{2\pi i}\frac{\dd}{\dd \tau} = q \frac{\dd}{\dd q}$. 
 These functions are weight $2\nu+k+l$ (possibly non-holomorphic) modular forms, which one can project to obtain a holomorphic modular form via an integral map $\pi_\mathrm{hol}.$
	
	Zagier studies the resulting sequence of modular forms $\pi_\mathrm{hol}([\mathcal{H}, \theta]_\nu|U_4)$,  where $\nu \geq 1$. He computes them in two ways. The first method is combinatorial, and it  uses the identity (for example, see ~\cite{Mertens2014})
\[
	\pi_\mathrm{hol}([\mathcal{H}, \theta]_\nu|U_4) = [\mathcal{H}^+, \theta]_\nu|U_4 + 2 \binom{2\nu}{\nu} \sum_{n=1}^\infty \lambda_{2\nu+1}(n) q^n.
\]
A straightforward brute force calculation with \eqref{RC} gives
\begin{align}\label{eq:RC-calc}
	[\mathcal{H}^+, \theta]_\nu |U_4 = \binom{2\nu}{\nu} \sum_{n=0}^\infty \left(\sum_{r \in \Z} p_{2\nu+2}(r,n) H(4n-r^2) \right) q^n.
\end{align}
Therefore, the $n$th coefficient of $\pi_\mathrm{hol}([\mathcal{H}, \theta]_\nu|U_4)$ is
\begin{align}\label{ES-LHS}
	\binom{2\nu}{\nu} \left(\sum_{r \in \Z} p_{2\nu+2}(r,n) H(4n-r^2) + 2\lambda_{2\nu+1}(n) \right).
\end{align}

As an alternate calculation, Zagier combines (for example, see ~\cite[Theorem 5.5]{EichlerZagier1985}) the Rankin--Cohen bracket operators with Hecke--Petersson theory.
As each $\pi_\mathrm{hol}([\mathcal{H}, \theta]_\nu |U_4)$ is a cusp form, we have
$$
	\pi_\mathrm{hol}([\mathcal{H}, \theta]_\nu |U_4) = \sum_{j=1}^{d_{2\nu+2}} a_j f_j,
$$	
where the $f_j$'s form a basis of Hecke eigenforms for $S_{2\nu+2}$.  In particular, we have $T_n f_j = c_{f_j}(n) f_j,$ 
where 
$$f_j(\tau)=q+\sum_{n\geq 2} c_{f_j}(n)q^n.$$

To compute the $a_j,$ he expresses $\mathcal{H}(\tau)$ in terms of Eisenstein series (see~\cite[Section 2.2]{OnoSaad2022} or \cite[Ch.~2]{HirzebruchZagier1976}), which allows him to use the method of unfolding and the Rankin--Selberg method to derive the Petersson inner product identity (for example, see~\cite[Ch.~6.3]{BFOR2017})
$$
	a_j \langle f_j, f_j \rangle = \langle \pi_\mathrm{hol}([\mathcal{H}, \theta]_\nu |U_4), f_j \rangle
		= -2 \binom{2\nu}{\nu} \langle f_j, f_j \rangle.
$$
For each $j$, this gives $a_j=-2\binom{2\nu}{\nu}.$	
Therefore, the
 $n$th coefficient of $\pi_\mathrm{hol}([\mathcal{H}, \theta]_\nu |U_4)$ is
$-2 \binom{2\nu}{\nu}\cdot  \mathrm{Tr}(n; 2\nu+2),$ which when equated with \eqref{ES-LHS} gives
the Eichler--Selberg trace formula.

\subsection{Proofs of Theorems 1.1-1.3}

Zagier's proof begins with the fact that $\mathcal{H}^+(\tau)$ is the holomorphic part of a weight 3/2 harmonic Maass form. 
In 2002, Zagier~\cite{Zagier2002} greatly generalized this fact.

\begin{theorem*}[Theorem 5 of \cite{Zagier2002}]\label{thm:Zagier-gm} For positive integers $m$, we have that
	\begin{equation}\label{gm}
		g_m(\tau)  \coloneqq  -\sum_{\kappa \mid m} \kappa q^{-\kappa^2} + 2\sigma_1(m) + \sum_{\substack{d > 0 \\ d \equiv 0, 3\ (\rm{mod}\ 4)}} \bt_m(d) q^d
	\end{equation}
	is a weakly holomorphic modular form of weight $3/2$ on $\Gamma_0(4)$.
	\end{theorem*}

\begin{proof}[Proof of \cref{FundGF}]

Emulating Zagier's proof of the Eichler--Selberg trace formula, we replace $\mathcal{H}^+(\tau)$ in \eqref{eq:RC-calc} with the $g_m(\tau)$. Namely, we define
 $$
\calG_{m,\nu}(\tau) \coloneqq -\frac{1}{2\binom{2\nu}{\nu}} \cdot [g_m, \theta]_\nu |U_4.
 $$
By the combinatorial calculation that gave \eqref{eq:RC-calc}, we obtain the earlier definition \eqref{KeyGF}
$$
	\calG_{m,\nu}(\tau) = -\frac{1}{2} \sum_{n \gg -\infty} \sum_{r \in \Z} p_{2\nu+2}(r,n) \bt_m(4n-r^2) q^n.
$$
Furthermore, the theory of Rankin--Cohen brackets in this setting (see ~\cite[Theorem 5.5]{EichlerZagier1985}) implies that $\calG_{m,\nu}(\tau)$ is a weakly holomorphic modular form in $M_{2\nu+2}^!.$
\end{proof}

\begin{proof}[Proof of \cref{EasyEichlerSelberg_nu0}]
The space of weight 2 holomorphic modular forms is $M_2 = \{0\}$ and 
$$Dj_{-n}(\tau) = n q^n + O(q) \in M_2^!.
$$
Therefore, we have
$$
	\calG_{m,0}(\tau) +\frac{1}{2} \sum_{-\frac{m^2}{4} \le n < 0} \frac{1}{n} \left(\sum_{r \in \Z} \bt_m(4n-r^2)\right) Dj_{-n}(\tau) = 0.
$$
The first claim follows from \eqref{eq:tm-neg}. By comparing the $n$th coefficients, the second claim is obtained.
\end{proof}

\begin{proof}[Proof of~\cref{EasyEichlerSelberg}]
For $\nu > 0$, we note that
\begin{align}\label{def:Gmnu}
	\calG_{m,\nu}(\tau) + \frac{1}{2} \sum_{-\frac{m^2}{4} \le n \le 0} \sum_{r \in \Z} p_{2\nu+2}(r,n) \bt_m(4n-r^2) P_{2\nu+2, n}(\tau)
\end{align}
is a  cusp form. We are merely cancelling the poles at infinity with Poincar\'e series that satisfy
\eqref{PP}, and we capture the constant term with Eisenstein series $P_{2\nu+2,0}(\tau)=E_{2\nu+2}(\tau)=1+\cdots.$ For $\nu \in \{1, 2, 3, 4, 6\}$, the space of cusp forms $S_{2\nu+2}=\{0\}$ is trivial. Therefore, the theorem follows from the identity
	\begin{align}\label{eq:pkr-exp-eq}
		p_{2\nu+2} \left(r, \frac{r^2-\kappa^2}{4}\right) = \frac{(\kappa - r)^{2\nu+1} + (\kappa + r)^{2\nu+1}}{2^{2\nu+1} \kappa}.
	\end{align}
\end{proof}

\section{Vector-valued modular forms of dimension 2}\label{Section3}

The proof of \cref{Main-Theorem} is much more involved than the proofs of Theorems 1.2 and 1.3. Nevertheless, its proof is still based on \cref{FundGF}, and the aim is to understand the Fourier expansion of $\calG_{m,\nu}(\tau)$ arithmetically in terms of traces of Hecke operators  and shifted convolution $L$-functions. These calculations shall depend on the arithmetic of half-integral weight vector-valued modular forms that transform with respect to the Weil representation. To this end, here we recall essential preliminaries.

\subsection{The Weil representation}

Let $\calO(\bbH)$ be the set of all holomorphic functions $\phi : \bbH \to \C$. For $z \in \C \setminus \{0\}$, we take the principal branch of $z^{1/2}$ as $\arg(z^{1/2}) \in (-\pi/2, \pi/2]$. For an integer $k \in \Z$, we put $z^{k/2} = (z^{1/2})^k$. For $n \in \Z_{\ge 0}$, we put $x^{\ov{n}} \coloneqq \Gamma(x+n)/\Gamma(x) = x(x+1)\cdots (x+n-1)$, and $x^{\un{n}} \coloneqq \Gamma(x+1)/\Gamma(x-n+1) = x(x-1)\cdots(x-n+1)$.

	The \emph{metaplectic group} $\Mp_2(\R)$ is a group defined by
	\[
		\Mp_2(\R) \coloneqq \left\{ (\gamma, \phi(\tau)) : \gamma = \pmat{a & b \\ c & d} \in \SL_2(\R), \phi \in \calO(\bbH)\text{ satisfying } \phi(\tau)^2 = c\tau+d \right\},
	\]
	where the group operation is
		$(\gamma_1, \phi_1(\tau)) \cdot (\gamma_2, \phi_2(\tau)) \coloneqq (\gamma_1 \gamma_2, \phi_1(\gamma_2 \tau) \phi_2(\tau)).$
		
As usual, we have $\gamma \tau \coloneqq \frac{a\tau+b}{c\tau+d}$, and
 for any $\gamma = \smat{a & b \\ c & d} \in \SL_2(\R)$, we define $j(\gamma, \tau) = c\tau+d$ and $\widetilde{\gamma} = \left(\smat{a & b \\ c & d}, j(\gamma, \tau)^{1/2} \right) \in \Mp_2(\R)$. Let $\Mp_2(\Z)$ be the inverse image of $\SL_2(\Z)$ under the projection $\Mp_2(\R) \to \SL_2(\R)$. As usual, we let
	$T \coloneqq \left( \begin{smallmatrix}1 & 1 \\ 0 & 1\end{smallmatrix}\right)$ and  $S \coloneqq \left( \begin{smallmatrix}0 & -1 \\ 1 & 0\end{smallmatrix}\right).$
It is well known that $\Mp_2(\Z)$ is generated by $\widetilde{T}$ and $\widetilde{S}$, (see~\cite[p.16]{Bruinier2002}) and its center is generated by
\[
	\widetilde{-I} = \widetilde{S}^2 = (\widetilde{S} \widetilde{T})^3 = \left(\pmat{-1 & 0 \\ 0 & -1}, i \right).
\]		
Moreover, we let $\widetilde{\Gamma}_\infty \coloneqq \langle \widetilde{T}\rangle \times \langle \widetilde{-I} \rangle$, representing the metaplectic stabilizer for the cusp at infinity.

We recall the {\it Weil representation}\footnote{
For more general settings, see Bruinier's book~\cite[Ch.~1]{Bruinier2002} and Borcherds~\cite{Borcherds1998}.}, the unitary representation $\rho : \Mp_2(\Z) \to \GL_2(\C)$ defined by
	\begin{equation}
	\rho(\widetilde{T}) \coloneqq \pmat{1 & 0 \\ 0 & i} \quad \text{ and } \quad \rho(\widetilde{S}) \coloneqq \frac{1}{\sqrt{2i}} \pmat{1 & 1 \\ 1 & -1}.
	\end{equation}
We note that $\rho(\widetilde{-I}) = \rho(\widetilde{S}^2) = -i I$. We let $\rho^* : \Mp_2(\Z) \to \GL_2(\C)$ be the dual representation of $\rho$
\[
	\rho^*((\gamma, \phi)) \coloneqq {}^t \rho((\gamma, \phi))^{-1} = \overline{\rho((\gamma, \phi))}.
\]

We recall an explicit formula for $\rho(\widetilde{\gamma})$, which is easily derived from work of both Shintani~\cite[Proposition 1.6]{Shintani1975} and
Bruinier~\cite[Proposition 1.1]{Bruinier2002}, where for odd integers $d$ we let
\begin{align}\label{def:epsilon}
	\epsilon_d \coloneqq \begin{cases}
		1 &\text{if } d \equiv 1 \pmod{4},\\
		i &\text{if } d \equiv 3 \pmod{4}.
	\end{cases}
\end{align}

\begin{proposition}\label{prop:Weil-rep-formula}
	For $c \ge 0$, we have
	\begin{align*}
		\rho(\widetilde{\gamma}) = \begin{cases}
			\displaystyle{\frac{\epsilon_c}{1+i} \left(\frac{a}{c}\right) \pmat{1 & i^{cd} \\ i^{ac} & -i^{(a+d)c}}} &\text{if } c \equiv 1 \pmod{2},\\
			\displaystyle{\epsilon_a^{-1} \left(\frac{c}{a}\right) \pmat{0 & i^{ab} \\ 1 & 0}} &\text{if } c \equiv 2 \pmod{4},\\
			\displaystyle{\epsilon_a^{-1} \left(\frac{c}{a}\right) \pmat{1 & 0 \\ 0 & i^{ab}}} &\text{if } c \equiv 0 \pmod{4}.
		\end{cases}
	\end{align*}
\end{proposition}

We now give the definition of a vector-valued modular form that transforms under the Weil representation.
If $k \in \frac{1}{2} \Z$ and $f: \bbH \to \C^2$. For $(\gamma, \phi(\tau)) \in \Mp_2(\Z),$ then we define the \emph{slash operator}
	\begin{align*}
		(f|_{k,\rho} (\gamma, \phi)) (\tau) \coloneqq \phi(\tau)^{-2k} \rho((\gamma, \phi))^{-1} f(\gamma \tau).
	\end{align*}
	We say that $f: \bbH \to \C^2$ is a \emph{weight $k$ (vector-valued) modular form with respect to $\rho$} if
	\[
		f|_{k,\rho} (\gamma, \phi) = f
	\]
	for every $(\gamma, \phi) \in \Mp_2(\Z)$. We define them for $\rho^*$ in a similar manner.

\subsection{Jacobi's theta functions}

For later use, we recall the Jacobi theta functions (for example, see~\cite[Section 5]{EichlerZagier1985}) in this context.
	If we set $\zeta \coloneqq \e(z)$, where $\e(z) \coloneqq e^{2\pi i z},$  we have
	\begin{equation}\label{def:Jacobi-theta}
		\theta_0(\tau, z) \coloneqq \sum_{\substack{r \in \Z \\ r \equiv 0\ (2)}} q^{r^2/4} \zeta^r \ \ \ \
		{\text {\rm and}}\ \ \ \  \theta_1(\tau, z) \coloneqq \sum_{\substack{r \in \Z \\ r \equiv 1\ (2)}} q^{r^2/4} \zeta^r
	\end{equation}
	and $\Theta(\tau, z) \coloneqq \left( \begin{smallmatrix}\theta_0(\tau, z) \\ \theta_1(\tau,z)\end{smallmatrix}\right).$
The specialization $\Theta(\tau, 0)$ is 2 dimensional weight 1/2 vector-valued modular form with respect to $\rho$,  and in general is a (vector-valued) Jacobi form, which for $(\gamma, \phi)\in \Mp_2(\Z),$ in this case means that
\begin{align}\label{eq:Jacobi-trans}
	(\Theta|_{1/2, 1, \rho} (\gamma, \phi))(\tau, z) \coloneqq \phi(\tau)^{-1} \e \left(\frac{-cz^2}{c\tau+d}\right) \rho((\gamma, \phi))^{-1} \Theta \left(\frac{a\tau+b}{c\tau+d}, \frac{z}{c\tau+d}\right) = \Theta(\tau, z).
\end{align}

\section{Maass--Poincar\'{e} series and traces of singular moduli}\label{Section4}

The proof of \cref{Main-Theorem} relies on Maass--Poincar\'e series that transform with respect to the Weil representation. We construct these series following \cite{Bruinier2002}, and we relate them to traces of singular moduli.

\subsection{The Whittaker functions}

Let $M_{\mu, \nu}(z)$ and $W_{\mu, \nu}(z)$ be the Whittaker functions (for example, see~\cite[Ch.~16]{WhittakerWatson1962} and  \cites{MOS1966, NIST}). The next two lemmas are crucial for constructing Maass--Poincar\'e series.

\begin{lemma}[{\cite[7.2.1]{MOS1966}, \cite[13.15.19]{NIST}}]\label{lem:M-Whit-der} For positive integers $n,$ we have
	\[
		\frac{\dd^n}{\dd z^n} \left(e^{-z/2} z^{-\nu-1/2} M_{\mu, \nu}(z) \right) = (-1)^n \frac{(\mu + \nu + 1/2)^{\ov{n}}}{(2\nu+1)^{\ov{n}}} e^{-z/2} z^{-\nu - n/2-1/2} M_{\mu+n/2, \nu+n/2}(z).
	\]
\end{lemma}

\begin{lemma}[{\cite[7.5.1]{MOS1966}, \cite[13.23.1]{NIST}}]\label{lem:M-Whit-int}
	For $\Re(\nu + \alpha + 1/2) > 0$ and $2\Re(z) > \beta > 0$, we have
	\begin{align*}
		\int_0^\infty e^{-zt} t^{\alpha-1} M_{\mu, \nu}(\beta t) \dd t &= \frac{\beta^{\nu+1/2} \Gamma\left(\nu+\alpha+\frac{1}{2}\right)}{\left(z+\frac{\beta}{2}\right)^{\nu+\alpha+1/2}}\cdot  {}_2F_1 \left(\nu-\mu+\frac{1}{2}, \nu+\alpha+\frac{1}{2}; 2\nu+1; \frac{2\beta}{\beta+2z}\right),
	\end{align*}
	where ${}_2F_1(a,b; c; z)$ is the Gaussian hypergeometric function.
\end{lemma}

For $n \in \Z$, $k \in \frac{1}{2} \Z$, $y > 0$, and $s \in \C$, we define the modified Whittaker functions
\begin{align}\label{def:MW-cal}
	\calM_{k,n}(y,s) &\coloneqq \begin{cases}
		\Gamma(2s)^{-1} (4\pi |n|y)^{-k/2} M_{\sgn(n) \frac{k}{2}, s-1/2}(4\pi |n|y) &\text{if } n \neq 0,\\
			y^{s-k/2} &\text{if } n = 0,
	\end{cases}\\
	\calW_{k,n}(y,s) &\coloneqq \begin{cases}
		\Gamma(s+ \sgn(n) \frac{k}{2})^{-1} |n|^{k-1} (4\pi |n| y)^{-k/2} W_{\sgn(n) \frac{k}{2}, s-1/2}(4\pi|n|y) &\text{if } n \neq 0,\\
		\dfrac{(4\pi)^{1-k} y^{1-s-k/2}}{(2s-1) \Gamma(s-k/2) \Gamma(s+k/2)} &\text{if } n = 0.
	\end{cases}
\end{align}
The special values of these functions at $s = k/2$ play a crucial role in the construction of the Maass--Poincar\'e series. To this end, for $n < 0$, we have
\begin{align}\label{eq:M-special}
	\calM_{k,n} \left(y, \frac{k}{2}\right) &= \Gamma(k)^{-1} e^{-2\pi ny}.
\end{align}
As for the $\mathcal{W}$-function, we have
\begin{align}\label{eq:W-special}
	\calW_{k,n} \left(y, \frac{k}{2}\right) = \begin{cases}
		\Gamma(k)^{-1} n^{k-1} e^{-2\pi ny} &\text{if } n > 0,\\
		0 &\text{if } n \le 0,
	\end{cases}
\end{align}
(see~\cite[7.2.4]{MOS1966}). Moreover, we note that, (\cite[7.6.1]{MOS1966}, \cite[13.14]{NIST}),
\begin{align}\label{eq:W-asymp}
\begin{split}
	W_{\mu, \nu}(y) &\sim e^{-y/2} y^\mu \quad (y \to \infty),\\
	M_{\mu, \nu}(y) &= y^{\nu+1/2} (1+ O(y)) \quad (y \to 0).
\end{split}
\end{align}

\subsection{Kloosterman sums}

The Fourier expansions of the Maass--Poincar\'e series require Kloosterman sums, which we recall here.
	For $k \in \frac{1}{2}\Z \setminus \Z$, $m, n \in \Z$, and $c>0$ with $c \equiv 0 \pmod{4}$, we define the half-integral weight \emph{Kloosterman sum} by
\begin{equation}\label{def:Kloosterman-sum}
		K_k(m,n,c) \coloneqq \sum_{d\ (c)^*} \left(\frac{c}{d}\right) \epsilon_d^{2k} \e \left(\frac{m \overline{d} + nd}{c}\right),
\end{equation}
	where $\overline{d} \in \Z/c\Z$ satisfies that $d \overline{d} \equiv 1 \pmod{c}$. The condition $d\ (c)^*$ means that $d$ runs over $d \in \Z/c\Z$ such that $(c,d) = 1$.
We note that the Kloosterman sums satisfy
\begin{align}\label{eq:Kloosterman}
	K_{k+2}(m,n,c) = K_k(m,n,c) \quad \text{ and } \quad K_{3/2}(m,n,c) = -i K_{1/2}(-m,-n,c).
\end{align}

We now relate the Weil representation to such Kloosterman sums. For notational convenience, we let
\[
	\rho(\widetilde{\gamma}) = \pmat{\rho(\widetilde{\gamma})_{00} & \rho(\widetilde{\gamma})_{01} \\ \rho(\widetilde{\gamma})_{10} & \rho(\widetilde{\gamma})_{11}}.
\]
Then the following sum formula holds for each entry of $\rho(\widetilde{\gamma})$.

\begin{proposition}\label{conj-Kloosterman}
	If $\alpha, \beta \in \{0,1\}$ and  $m$ and $n$ satisfy $m \equiv -\alpha \pmod{4}$ and $n \equiv -\beta \pmod{4},$ then for every positive integer  $c$ we have
	$$
	\frac{1}{4} \left(1 + \left(\frac{4}{c}\right) \right) K_{3/2}(m,n,4c)=
	\sum_{d\ (c)^*} \rho(\widetilde{\gamma})_{\alpha \beta} \e\left(\frac{ma + nd}{4c}\right),
	$$
	where we take any $\gamma = \smat{a & b \\ c & d} \in \SL_2(\Z)$ for which $(c,d)$ forms its bottom row.
\end{proposition}

\begin{proof}
	First, we check that the right-hand side is well-defined. Let $R_{\alpha \beta}(\gamma)$ denote its summand. It suffices to show that $R_{\alpha \beta}(T^j \gamma T^l) = R_{\alpha \beta}(\gamma)$ holds for any $j, l \in \Z$. Since $\widetilde{T^j \gamma T^l} = \widetilde{T}^j \widetilde{\gamma} \widetilde{T}^l$ holds, we have
	\[
		R_{\alpha \beta}(T^j \gamma T^l) = i^{\alpha j + \beta l} \rho(\widetilde{\gamma})_{\alpha \beta} \e \left(\frac{ma+nd}{4c}\right) \e \left(\frac{mj+nl}{4}\right) = R_{\alpha \beta}(\gamma).
	\]
	Next, for each $\gamma = \smat{a & b \\ c & d} \in \SL_2(\Z)$ with $c > 0$, we prove the refined equation
	\begin{align}\label{eq:refined}
		\rho(\widetilde{\gamma})_{\alpha \beta} = \frac{1}{4} \left(1 + \left(\frac{4}{c}\right) \right) \sum_{\substack{\delta\ (4c) \\ \delta \equiv d\ (c) \\ \delta \equiv 1\ (2)}} \left(\frac{c}{\delta}\right) \epsilon_{\delta}^{-1} \e \left(\frac{a - \overline{\delta}}{4c}\right)^\alpha \e \left(\frac{d-\delta}{4c}\right)^\beta,
	\end{align}
	where $\overline{\delta}$ is the inverse of $\delta$ in $(\Z/4c\Z)^\times$. This immediately implies the proposition. 
	
	To confirm \eqref{eq:refined},  let $\delta = d_j \coloneqq d + cj$ and $b_j \coloneqq b+aj$ ($j = 0,1,2,3$). For simplicity, let $\rho''(\gamma)_{\alpha \beta}$ denote the right-hand side of \eqref{eq:refined} and show that $\rho''(\gamma)_{\alpha \beta} = \rho(\widetilde{\gamma})_{\alpha \beta}$. If $\delta$ is odd, then we can easily check that 
	\[
		\overline{\delta} = \begin{cases}
			a(1+b_j c) &\text{if } c \equiv 1 \pmod{2} \text{ and } a \equiv 1 \pmod{2},\\
			(a+c)(1+(b_j+d_j)c) &\text{if } c \equiv 1 \pmod{2} \text{ and } a \equiv 0 \pmod{2},\\
			a(1-ab_j cd_j) &\text{if } c \equiv 2 \pmod{4},\\
			a(1-b_jc) &\text{if } c \equiv 0 \pmod{4}.
		\end{cases}
	\]
We prove the case where $c \equiv 1 \pmod{2}$ and $a \equiv 1 \pmod{2}$, leaving the others to the reader.
We have
	\begin{align*}
		\rho''(\gamma)_{\alpha \beta} &= \frac{1}{2} \sum_{\substack{0 \le j \le 3 \\ d_j \equiv 1\ (2)}} \left(\frac{c}{d_j}\right) \epsilon_{d_j}^{-1} \e \left(\frac{-ab_j}{4}\right)^\alpha \e \left(\frac{-j}{4}\right)^\beta
		= \frac{i^{-ab \alpha}}{2} \left(\frac{d}{c}\right) \sum_{\substack{0 \le j \le 3 \\ d_j \equiv 1\ (2)}} (-1)^{\frac{(c-1)(d_j-1)}{4}} \epsilon_{d_j}^{-1} i^{-j(\alpha+\beta)}.
	\end{align*}
	Since the value of the sum depends only on $c, d \pmod{4}$, a direct calculation yields 
	\[
		\rho''(\gamma)_{\alpha \beta} = \frac{i^{-ab \alpha}}{2} \left(\frac{d}{c}\right) \frac{2\epsilon_c}{1+i} \times \begin{cases}
			1 &\text{if } (\alpha, \beta) = (0,0),\\
			i^{cd} &\text{if } (\alpha, \beta) = (0,1), (0,1),\\
			(-1)^{d-1} &\text{if } (\alpha, \beta) = (1,1).
		\end{cases}
	\]
	Combining simple calculations with \cref{prop:Weil-rep-formula}, one obtains $\rho(\widetilde{\gamma})_{\alpha \beta}$.
\end{proof}

\subsection{The Maass--Poincar\'{e} series}
Using the  two previous subsections, we now construct  the Maass--Poincar\'e series.
We let $\fe_0 \coloneqq \smat{1 \\ 0}$ and $\fe_1 \coloneqq \smat{0 \\ 1}$.
	Assume that $k \in \frac{1}{2} \Z$ satisfies $2k \equiv 3 \pmod{4}$. For $\alpha \in \{0, 1\}$ and $m \equiv -\alpha \pmod{4}$, we define the \emph{Maass--Poincar\'{e} series} of weight $k$ with respect to $\rho^*$ by
	\begin{align}\label{def:PoincareSeries}
		P_{k, \rho^*}^{(\alpha, m)} (\tau, s) &\coloneqq \sum_{(\gamma, \phi) \in \widetilde{\Gamma}_\infty \backslash \Mp_2(\Z)} \calM_{k,m} \left(\frac{v}{4}, s\right) \e \left(\frac{mu}{4}\right) \fe_\alpha \bigg|_{k,\rho^*} (\gamma, \phi)\\
			&\ = \sum_{\gamma \in \Gamma_\infty \backslash \Gamma} \calM_{k,m} \left(\frac{v}{4}, s\right) \e\left(\frac{mu}{4}\right) \fe_\alpha \bigg|_{k,\rho^*} \widetilde{\gamma}.\notag
	\end{align}
	This series converges absolutely and uniformly on compact subsets in $\Re(s) > 1$ \cite[p.29]{Bruinier2002}, and we note that $\calM_{k,m}(v/4, s) \e(mu/4) \fe_\alpha$ is invariant under  $|_{k,\rho^*} (\gamma, \phi)$ for any $(\gamma, \phi) \in \widetilde{\Gamma}_\infty$ as $2k \equiv 3 \pmod{4}$. 
	
	The Fourier expansions of the functions involve
the Bessel functions (see \cite[Ch.~3]{MOS1966} and \cite[Ch.~17]{WhittakerWatson1962})
\begin{align*}
	I_\nu(z) \coloneqq \sum_{m=0}^\infty \frac{(z/2)^{\nu+2m}}{m! \Gamma(\nu+m+1)}, \quad J_\nu(z) \coloneqq \sum_{m=0}^\infty \frac{(-1)^m (z/2)^{\nu+2m}}{m! \Gamma(\nu+m+1)}.
\end{align*}

\begin{proposition}\label{Thm:Fourier-exp}
	For $\Re(s) > 1$, we have
	\begin{align*}
		P_{k,\rho^*}^{(\alpha, m)}(\tau, s) = \calM_{k,m} \left(\frac{v}{4},s\right) \e\left(\frac{mu}{4}\right) \fe_\alpha + \sum_{\beta \in \{0,1\}} \sum_{\substack{n \in \Z \\ n \equiv -\beta\ (4)}} b_{m,k}^{(\beta)}(n,s) \calW_{k,n} \left(\frac{v}{4}, s\right) \e\left(\frac{nu}{4}\right) \fe_\beta,
	\end{align*}
	where
	\begin{align*}
		b_{m,k}^{(\beta)}(n, s) &= 2\pi i^{-k} \sum_{c > 0} \left(1+\left(\frac{4}{c}\right)\right) \frac{K_{3/2}(m,n,4c)}{4c}\\
		&\qquad \qquad \times \begin{cases}
		\displaystyle{|mn|^{\frac{1-k}{2}} J_{2s-1} \left(\frac{\pi \sqrt{|mn|}}{c}\right)} &\text{if } mn > 0,\\
		\displaystyle{|mn|^{\frac{1-k}{2}} I_{2s-1} \left(\frac{\pi \sqrt{|mn|}}{c}\right)} &\text{if } mn < 0,\\
		\displaystyle{2^{k-1} \pi^{s+k/2-1} |m+n|^{s-k/2} (4c)^{1-2s}} &\text{if } mn=0, m+n \neq 0,\\
		\displaystyle{2^{2k-2} \pi^{k-1} \Gamma(2s) (8c)^{1-2s}} &\text{if } m = n = 0.
	\end{cases}
	\end{align*}
\end{proposition}

\subsection{Proof of \cref{Thm:Fourier-exp}}

Following the method in~\cite[Section 1.3]{Bruinier2002}, we have
\begin{align*}
	P_{k, \rho^*}^{(\alpha, m)} (\tau, s) &= \calM_{k,m} \left(\frac{v}{4}, s\right) \e\left(\frac{mu}{4}\right) \fe_\alpha + \sum_{\substack{\gamma = \smat{a & b \\ c & d} \in \Gamma_\infty \backslash \Gamma \\ c > 0}} \calM_{k,m} \left(\frac{v}{4}, s\right) \e\left(\frac{mu}{4}\right) \fe_\alpha \bigg|_{k,\rho^*} \widetilde{\gamma}.
\end{align*}
Let $H_{k,\rho^*}^{(\alpha, m)}(\tau, s)$ denote the sum of the second term, which is calculated as follows:
\begin{align*}
	H_{k,\rho^*}^{(\alpha, m)}(\tau,s) &= \sum_{\substack{\gamma \in \Gamma_\infty \backslash \Gamma \\ c > 0}} j(\gamma, \tau)^{-k} \calM_{k,m}\left(\frac{\Im (\gamma \tau)}{4}, s\right) \e \left(\frac{m \Re (\gamma \tau)}{4} \right) \rho^* (\widetilde{\gamma})^{-1} \fe_\alpha\\
	&= \sum_{\substack{\gamma \in \Gamma_\infty \backslash \Gamma/\Gamma_\infty \\ c >0}} \sum_{h \in \Z} j(\gamma, T^h \tau)^{-k} \calM_{k,m}\left(\frac{\Im(\gamma T^h \tau)}{4}, s\right) \e\left(\frac{m \Re(\gamma T^h \tau)}{4}\right) \pmat{1 & 0 \\ 0 & i^h} {}^t \rho (\widetilde{\gamma}) \fe_\alpha.
\end{align*}
Here we used the facts that $j(\gamma T^h, \tau) = j(\gamma, T^h \tau)$, $\widetilde{\gamma T^h} = \widetilde{\gamma} \widetilde{T}^h$, and
\[
	\rho^*(\widetilde{T}^h)^{-1} = {}^t \rho(\widetilde{T})^h = \pmat{1 & 0 \\ 0 & i^h}.
\]
Since the inner sum over $h \in \Z$ is invariant under $\tau \to \tau+4$, we obtain the Fourier expansion
\[
	H_{k,\rho^*}^{(\alpha, m)}(\tau,s) = \sum_{\substack{\gamma \in \Gamma_\infty \backslash \Gamma/\Gamma_\infty \\ c>0}} \sum_{\beta \in \{0,1\}} \sum_{n \in \Z} c_{\beta, \gamma}(n,v) \e \left(\frac{nu}{4}\right) \fe_\beta,
\]
where the sum over $\gamma \in \Gamma_\infty \backslash \Gamma / \Gamma_\infty$ with $c > 0$ is equivalent to the sum over $c > 0$ and $d \in (\Z/c\Z)^\times$ and 
\begin{align*}
	&c_{\beta, \gamma}(n, v)\\
	&\ \ \ = \frac{1}{4} \int_0^4 \e\left(\frac{-nu}{4}\right) {}^t \fe_\beta \sum_{h \in \Z} j(\gamma, T^h \tau)^{-k} \calM_{k,m}\left(\frac{\Im(\gamma T^h \tau)}{4}, s\right) \e\left(\frac{m \Re(\gamma T^h \tau)}{4}\right) \pmat{1 & 0 \\ 0 & i^h} {}^t \rho (\widetilde{\gamma}) \fe_\alpha \dd u.
\end{align*}
Using $\epsilon_d$ introduced in \eqref{def:epsilon}, we find that
\begin{align*}
	&= \frac{\rho(\widetilde{\gamma})_{\alpha \beta}}{4} \sum_{h =0}^3 \epsilon_{(-1)^\beta}^h \e \left(\frac{nh}{4}\right) \int_{-\infty}^\infty \e \left(\frac{-nu}{4}\right) j(\gamma, \tau)^{-k} \calM_{k,m} \left(\frac{\Im(\gamma \tau)}{4}, s\right) \e\left(\frac{m \Re(\gamma \tau)}{4}\right) \dd u\\
	&= \begin{cases}
		\displaystyle{\rho(\widetilde{\gamma})_{\alpha \beta} \int_{-\infty}^\infty \e\left(\frac{-nu}{4}\right) j(\gamma, \tau)^{-k} \calM_{k,m}\left(\frac{\Im(\gamma \tau)}{4}, s\right) \e\left(\frac{m \Re(\gamma \tau)}{4}\right) \dd u} &\ \ \ \text{if } n \equiv -\beta \pmod{4},\\
		0 &\ \ \ \text{if otherwise}.
	\end{cases}
\end{align*}
Next, we compute the integral. Assume that $n \equiv -\beta \pmod{4}$. Since we have
\[
	\gamma \tau = \frac{a}{c} - \frac{u + d/c}{c^2|\tau+d/c|^2} + i \frac{v}{c^2|\tau+d/c|^2}
\]
for $\gamma \in \SL_2(\R)$, we find that
\begin{align*}
	&c_{\beta, \gamma}(n,v)
	= \rho(\widetilde{\gamma})_{\alpha \beta} \int_{-\infty}^\infty c^{-k} (\tau+d/c)^{-k} \calM_{k,m} \left(\frac{v}{4c^2|\tau+d/c|^2}, s\right) \e\left(\frac{m}{4} \left(\frac{a}{c} - \frac{u + d/c}{c^2|\tau+d/c|^2} \right) - \frac{nu}{4}\right) \dd u.
\end{align*}
By changing variables via $u' = u+d/c$ and $\tau' = u'+iv$, we have
\begin{align*}
	&= \rho(\widetilde{\gamma})_{\alpha \beta} \e \left(\frac{ma+nd}{4c} \right) \int_{-\infty}^\infty c^{-k} \tau'^{-k} \calM_{k,m} \left(\frac{v}{4c^2|\tau'|^2}, s\right) \e\left(- \frac{mu'}{4c^2|\tau'|^2} - \frac{nu'}{4} \right) \dd u'.
\end{align*}
The last integral, denoted by $I_m(n)$, is calculated as follows (for example, see~\cite[Proof of Theorem 1.9]{Bruinier2002})
\begin{align*}
	I_m(n) = \begin{cases}
		\displaystyle{\frac{2 \pi i^{-k}}{c} |mn|^{\frac{1-k}{2}} J_{2s-1} \left(\frac{\pi \sqrt{|mn|}}{c}\right) \calW_{k,n}(v/4, s)} &\text{if } mn > 0,\\
		\displaystyle{\frac{2 \pi i^{-k}}{c} |mn|^{\frac{1-k}{2}} I_{2s-1} \left(\frac{\pi \sqrt{|mn|}}{c}\right) \calW_{k,n}(v/4, s)} &\text{if } mn < 0,\\
		\displaystyle{\frac{4^{1+k/2-2s} \pi^{s+k/2} i^{-k} |m+n|^{s-k/2}}{c^{2s}} \calW_{k,n}(v/4, s)} &\text{if } mn=0, m+n \neq 0,\\
		\displaystyle{\frac{4^{-3s+k+1} \pi^k i^{-k}}{c^{2s}} \Gamma(2s) \calW_{k,0}(v/4, s)} &\text{if } m = n = 0.
	\end{cases}
\end{align*}
By combining this with
\cref{conj-Kloosterman}, we obtain \cref{Thm:Fourier-exp}.

\subsection{Traces of singular moduli}\label{section:Traces}

The coefficients of these functions are related to traces of singular moduli, which is analogous to several previous works
(for example, see~\cite{BringmannOno2007, BruinierJenkinsOno2006, DukeImamogluToth2011}).
To make this precise, we consider weight 3/2 modular forms $h$ on $\Gamma_0(4)$ satisfying
\begin{align}\label{eq:Kohnen-plus}
	h(\tau) = \sum_{n \equiv 0,3\ (4)} c_n(v) q^n.
\end{align}
We define $h_i(\tau) = \sum_{n \equiv -i\ (4)} c_n(v/4) q^{n/4}$ for $i \in \{0, 1\}$, and then we have that
\begin{align}\label{eq:Kohnen-vector}
	H(\tau) \coloneqq \pmat{h_0(\tau) \\ h_1(\tau)} 
\end{align}
is a weight 3/2 vector-valued modular form with respect to $\rho^*$ (see~\cite[Section 5]{EichlerZagier1985} and \cite[Ch.~2]{BFOR2017}). 

We relate the $g_m(\tau)$ in \eqref{gm} and $g_0(\tau) \coloneqq \mathcal{H}(\tau)$ to the Maass--Poincar\'e expressions
\begin{align}
		G_m(\tau, s) \coloneqq \begin{cases}
			\displaystyle{-\frac{1}{12} P_{3/2, \rho^*}^{(0,0)} (\tau, s)} &\text{if } m=0,\\
			\displaystyle{-\frac{\sqrt{\pi}}{2} \sum_{n \mid m} n P_{3/2, \rho^*}^{(\alpha, -n^2)}(\tau, s) + 2 \sigma_1(m) P_{3/2, \rho^*}^{(0,0)} (\tau, s)} &\text{if } m > 0,
		\end{cases}
	\end{align}
	where $\alpha \equiv n^2 \pmod{4}$ for each $n$.  To be precise, we have the following theorem.

\begin{theorem}\label{Thm:construct:Gm}
	If $m$ is a non-negative integer, then we have
	\[
		\lim_{s \to 3/4} G_m(\tau, s) = \pmat{g_{m,0}(\tau) \\ g_{m,1}(\tau)}.
	\]
\end{theorem}
\begin{remark*}
	We note that the case of $m=0$ was stated by Williams~\cite[Example 5.1]{Williams2019}.
\end{remark*}
\begin{proof}[Sketch of the Proof] This result is standard, and so we sketch the proof.
We first recall facts about Niebur--Poincar\'{e} series  $F_m(\tau, s)$ (see~\cite{Niebur1973} or~\cite[Section 4]{DukeImamogluToth2011}), which are defined for $\Re(s) > 1,$  and give alternative expressions for the  $j_m(\tau)$. Specifically, as described in~\cite[(4.10)]{DukeImamogluToth2011}, it is known that
\[
	\mathop{\mathrm{Res}}_{s=1} F_0(\tau, s) = \frac{3}{\pi}
\]
and
\begin{equation}\label{jm_connection}
	\lim_{s \to 1} F_{-m}(\tau, s) = j_m(\tau) + 24 \sigma_1(m) \quad (m > 0).
\end{equation}
For non-negative integers $m$, the trace functions
\[
	\mathrm{Tr}_d(F_{-m}(\cdot, s)) \coloneqq \sum_{Q \in \mathcal{Q}_d/\Gamma} \frac{F_{-m}(\alpha_Q, s)}{\# \Gamma_Q}
\]
have a direct connection to the coefficients of the earlier Maass--Poincar\'{e} series.
Indeed, Duke,  Imamo\={g}lu, and T\'{o}th proved \cite[Proposition 4]{DukeImamogluToth2011},
	for $\Re(s) > 1$, $m \ge 0$, and $d > 0$ with $d \equiv 0, 3 \pmod{4}$, that
	\[
		\mathrm{Tr}_d(F_{-m}(\cdot, s)) = \begin{cases}
			\displaystyle{- d^{1/2} \sum_{n|m} n b_{-n^2, 3/2}^{(\beta)} \left(d, \frac{s}{2} + \frac{1}{4}\right)} &\text{if } m > 0,\\
			\displaystyle{-2^{s-2} \pi^{-s/2-1} d^{1/2} \zeta(s) b_{0, 3/2}^{(\beta)} \left(d, \frac{s}{2} + \frac{1}{4}\right)} &\text{if } m = 0.
		\end{cases}
	\]
	Therefore, \eqref{jm_connection} implies that
	\begin{equation}\label{DukeImamogluToth2011}
		\bt_m(d) = \lim_{s \to 3/4} \begin{cases}
			\displaystyle{- d^{1/2} \sum_{n|m} n b_{-n^2, 3/2}^{(\beta)} (d, s) + \frac{4d^{1/2}}{\sqrt{\pi}} \sigma_1(m) b_{0, 3/2}^{(\beta)}(d,s)} &\text{if } m > 0,\\
			\displaystyle{- \frac{d^{1/2}}{6 \sqrt{\pi}} b_{0, 3/2}^{(\beta)} (d, s)} &\text{if } m = 0.
		\end{cases}
	\end{equation}
By \eqref{eq:M-special} and \eqref{eq:W-special}, the formulas in \cref{Thm:Fourier-exp} equal these expressions, thereby proving the theorem. \end{proof}

\begin{remark*}
We note that subtle technicalities arise in the proof of \cref{Thm:construct:Gm},
 which have been addressed in the aforementioned works but deserve commentary. The $G_m(\tau, s)$ are defined for $\Re(s) > 1$, where they enjoy the Fourier series expansion in \cref{Thm:Fourier-exp}.  As  we can only be analytically continued up to $\Re(s) > 3/4$, care is required when letting $s \to 3/4$. In fact, the Fourier coefficients $b_{-m^2, 3/2}^{(\beta)}(-n^2, s)$ have a simple pole at $s = 3/4$, which cancels out with a zero from $\mathcal{W}_{3/2, -n^2}(v/4, s)$, (for example, see~\cite[Lemma 3]{DukeImamogluToth2011}). This issue is addressed by examining the growth of the Fourier coefficients of $G_m(\tau, s)$, including $\mathrm{Tr}_d(F_{-m}(\cdot, s))$, as $d \to \infty$ and the behavior as $s \to 3/4$. We refer the reader to \cite{BruinierJenkinsOno2006, Duke2006, DukeImamogluToth2011} for these details.
\end{remark*}

\section{Proof of \cref{Main-Theorem}}\label{Section5}

We have constructed the Poincar\'{e} series $G_m(\tau, s)$ whose Fourier coefficients give the traces of singular moduli. We turn to the problem of providing the Hecke decomposition of $\mathcal{G}_{m,\nu}(\tau)$.
Specifically, we compute the Petersson inner product $\langle \mathcal{G}_{m,\nu}, f \rangle$ with a normalized Hecke eigenform $f$ of $S_{2\nu+2}$. We first recall useful facts about Jacobi forms to relate the Rankin--Cohen brackets to these Poincar\'{e} series.

\subsection{Jacobi forms and the modified heat operator}\label{section:Jacobi}

For a function $\phi: \bbH \times \C \to \C$, $\gamma \in \SL_2(\Z)$, and positive integers $k, m \in \Z_{>0}$, we define the slash operator
\[
	(\phi|_{k,m} \gamma) (\tau, z) \coloneqq (c\tau+d)^{-k} \e \left(\frac{-cmz^2}{c\tau+d}\right) \phi \left(\frac{a\tau+b}{c\tau+d}, \frac{z}{c\tau+d}\right),
\]
and the {\it weighted heat operator}
\[
	L_{k,m} \coloneqq - D - \frac{1}{16\pi^2 m} \left(\frac{\partial^2}{\partial z^2} + \frac{2k-1}{z} \frac{\partial}{\partial z}\right).
\]
Then we have
\begin{align}\label{eq:Lk-slash}
	L_{k,m} (\phi|_{k,m} \gamma) = (L_{k,m} \phi)|_{k+2,m} \gamma
\end{align}
for any $\gamma \in \SL_2(\Z)$, (see~\cite[(11) in Section 3]{EichlerZagier1985}). For simplicity, we put $L_k \coloneqq L_{k,1}$.

\begin{lemma}\label{lem:Theta-Poincare}
	For a Poincar\'{e} series defined by
	\[
		G(\tau) = \sum_{(\gamma, \phi) \in \widetilde{\Gamma}_\infty \backslash \Mp_2(\Z)} \pmat{\varphi_0(\tau) \\ \varphi_1(\tau)} \bigg|_{3/2, \rho^*} (\gamma, \phi),
	\]
	we have
	\[
		{}^t \Theta(\tau, z) G(\tau) = \sum_{\gamma \in \Gamma_\infty \backslash \Gamma} (\theta_0(\tau, z) \varphi_0(\tau) + \theta_1(\tau, z) \varphi_1(\tau)) |_{2,1} \gamma.
	\]
\end{lemma}

\begin{proof}
	By a direct calculation with \eqref{eq:Jacobi-trans}, we have
	\begin{align*}
		{}^t \Theta(\tau, z) G(\tau)
		= \sum_{(\gamma, \phi) \in \widetilde{\Gamma}_\infty \backslash \Mp_2(\Z)} \phi(\tau)^{-4} \e \left(\frac{-cz^2}{c\tau+d}\right) {}^t \Theta\left(\gamma \tau, \frac{z}{c\tau+d}\right) {}^t \rho((\gamma, \phi))^{-1} \rho^*((\gamma, \phi))^{-1} \pmat{\varphi_0(\gamma\tau) \\ \varphi_1(\gamma \tau)}.
	\end{align*}
	Since ${}^t \rho((\gamma, \phi))^{-1} \rho^*((\gamma, \phi))^{-1} = I$ and $\phi(\tau)^{-4} = (c\tau+d)^{-2}$, we obtain the result.
\end{proof}

We require the following proposition for the $p_k(r,n)$  in the Eichler--Selberg trace formula.

\begin{proposition}\label{prop:LLL-p}
	For $\nu \ge 0$, we define
	\begin{align}\label{def:p-r-D-l}
		p_{2\nu+2}(r, D, l)  \coloneqq  \sum_{0 \le j \le \nu} (-1)^j \binom{2\nu+2l-j}{j} \frac{\binom{2l}{l} \binom{\nu+l-j}{l}}{\binom{2\nu+2l-j}{l} \binom{\nu+l}{l}} D^j r^{2\nu-2j}.
	\end{align}
	Then we have the Taylor expansion
	\[
		L_{2\nu} \circ \cdots \circ L_2 f(\tau)(\zeta^r + \zeta^{-r}) = 2 \sum_{l=0}^\infty p_{2\nu+2}(r,D,l) f(\tau) \frac{(2\pi irz)^{2l}}{(2l)!}.
	\]
	In particular, we have that
	\[
		\lim_{z \to 0} L_{2\nu} \circ \cdots \circ L_2 f(\tau)(\zeta^r + \zeta^{-r}) = 2 p_{2\nu+2}(r,D) f(\tau).
	\]
\end{proposition}

\begin{proof}
	We check that the Taylor coefficients of $L_{2\nu} \circ \cdots \circ L_2 f(\tau) (\zeta^r + \zeta^{-r})$ and the sequence~\eqref{def:p-r-D-l} satisfy the same recursion. The claim is clear for $\nu = 0$. For $\nu > 0$, let
	\[
		S_{\nu, l, j} \coloneqq (-1)^j \binom{2\nu+2l-j}{j} \frac{\binom{2l}{l} \binom{\nu+l-j}{l}}{\binom{2\nu+2l-j}{l} \binom{\nu+l}{l}} D^j r^{2\nu-2j}.
	\]
	Then $S_{\nu, l, j}$ satisfies the recursion
	\[
		S_{\nu, l, j} = -D S_{\nu-1, l, j-1} + \frac{r^2}{4} \left(1 + \frac{4\nu-1}{2l+1}\right) S_{\nu-1, l+1, j},
	\]
	for $\nu \ge 1$ and $0 \le j \le \nu$ with $S_{\nu, l, -1} = 0$, which implies that
	\[
		p_{2\nu+2}(r,D,l) = -D p_{2\nu}(r,D,l) + \frac{r^2}{4} \left(1 + \frac{4\nu-1}{2l+1}\right) p_{2\nu} (r,D,l+1).
	\]
One can check that the Taylor coefficients also satisfy this recursion.
\end{proof}

We use this proposition to understand the combinatorial properties of the Rankin--Cohen bracket operators, which is a slight generalization of \cite[Theorem 5.5]{EichlerZagier1985}.

\begin{proposition}\label{lem:key-fact}
	Let $\nu \ge 0$ be a non-negative integer. For a modular form $h$ of weight $3/2$ on $\Gamma_0(4)$ of the form~\eqref{eq:Kohnen-plus}, we have
	\begin{align*}
		[h, \theta]_\nu|U_4 = \binom{2\nu}{\nu} \sum_{n \in \Z} \sum_{r \in \Z} p_{2\nu+2}(r,D) c_{4n-r^2}(v/4) q^n = \binom{2\nu}{\nu} \lim_{z \to 0} L_{2\nu} \circ \cdots \circ L_2 {}^t \Theta(\tau, z) H(\tau).
	\end{align*}
\end{proposition}

\begin{proof}
	By definition, we have
	\begin{align*}
		[h, \theta]_\nu |U_4 &= \sum_{\substack{r,s \ge 0 \\ r+s = \nu}} (-1)^r \frac{\Gamma(3/2+\nu) \Gamma(1/2+\nu)}{s! \Gamma(3/2+r) r! \Gamma(1/2+s)} D^r\left(\sum_{n \equiv 0,3\ (4)} c_n(v) q^n \right) D^s \left(\sum_{m \in \Z} q^{m^2}\right) |U_4.
	\end{align*}
	A direct calculation implies that
	\begin{align*}
		D^r\left(\sum_{n \equiv 0,3\ (4)} c_n(v) q^n \right) D^s \left(\sum_{m \in \Z} q^{m^2}\right) |U_4 = \sum_{N \in \Z} \sum_{m \in \Z} m^{2s} (4D-m^2)^r c_{4N-m^2}(v/4) q^N
	\end{align*}
	and
	\begin{align*}
		\sum_{\substack{r, s \ge 0\\ r+s = \nu}} (-1)^r \frac{\Gamma(3/2+\nu) \Gamma(1/2+\nu)}{s! \Gamma(3/2+r) r! \Gamma(1/2+s)} m^{2s} (4D-m^2)^r = \binom{2\nu}{\nu} p_{2\nu+2}(m,D).
	\end{align*}
	The last equation immediately follows from \cref{prop:LLL-p}.
\end{proof}

For each $n \ge 0$, we define
\[
	\phi_{n, \nu}(\tau, s) \coloneqq \lim_{z \to 0} L_{2\nu} \circ \cdots \circ L_2 {}^t \Theta(\tau, z) P_{3/2, \rho^*}^{(\alpha, -n^2)}(\tau, s).
\]
Combining \cref{Thm:construct:Gm} and \cref{lem:key-fact}, for  $m\geq 1,$ we obtain the following key expressions
\begin{align}\label{eq:RC-heat}
	\calG_{m,\nu}(\tau) = -\frac{1}{2\binom{2\nu}{\nu}} \cdot [g_m, \theta]_\nu |U_4 = -\frac{1}{2} \lim_{s \to 3/4} \left( -\frac{\sqrt{\pi}}{2} \sum_{n \mid m} n \phi_{n, \nu}(\tau, s) + 2\sigma_1(m) \phi_{0,\nu}(\tau, s)\right).
\end{align}
The order of limits of $s$ and $z$ is interchanged, which is justified by the Remark at the end of~\cref{section:Traces}. 

\subsection{The Selberg--Poincar\'{e} series}\label{section:selberg-Poincare}

To prove \cref{Main-Theorem} using \eqref{eq:RC-heat}, we must
calculate $\phi_{n,\nu}(\tau, s)$ and $\langle \phi_{n, \nu}(\cdot, s), f\rangle$ at $s = 3/4$ for Hecke eigenforms $f$.
To this end, we use Selberg's generalization~\cite{Selberg1965} of the Poincar\'{e} series in \eqref{def:classical-Poincare}. For integers $k\ge 2$ and $m\in \Z,$ they are defined by
\begin{equation}
	P_{k,m}(\tau, s) \coloneqq \sum_{\gamma \in \Gamma_\infty \backslash \Gamma} v^s q^m |_k \gamma.
\end{equation}
This series converges absolutely and uniformly on compact subsets for $\Re(s) > 1-k/2$ and admits meromorphic continuation. In particular, it is known that $P_{k,m}(\tau, s)$ is holomorphic at $s = 1-k/2$. This fact follows from comparing it with the Maass--Poincar\'{e} series defined by
\[
	\sum_{\gamma \in \Gamma_\infty \backslash \Gamma} \mathcal{M}_{k,m}(v, s+k/2) \e(mu) |_k \gamma \quad (\Re(s) > 1-k/2).
\]
Indeed, from \eqref{eq:W-asymp}, we have
\[
	(4\pi|m|v)^s - \Gamma(2s+k) \mathcal{M}_{k,m}(v,s+k/2) = O(v^{\Re(s)+1}).
\]
Thus, for $\Re(s) > -k/2$, the poles of these two types of Poincar\'{e} series agree. On the other hand, the Fourier expansion of the Maass--Poincar\'{e} series, (see~\cite[Theorem 3.2]{JeonKangKim2013}), and the Weil bound for the Kloosterman sums imply its holomorphy at $s = 1-k/2$. 

The next lemma describes the Petersson inner product of cusp forms with these series.

\begin{lemma}\label{lem:Poincare-Petersson}
	For $f \in S_k$ and $m > 0$, we have
	\[
		\langle P_{k,m}(\cdot, s), f \rangle \coloneqq \int_{\Gamma \backslash \bbH} P_{k,m}(\tau, s) \overline{f(\tau)} v^k \frac{\mathrm{d}u \mathrm{d}v}{v^2} = \frac{\Gamma(s+k-1)}{(4\pi m)^{s+k-1}} \overline{c_f(m)}.
	\]
\end{lemma}

\begin{proof}
	It follows from the classical unfolding argument, (see~\cite[Ch.~10.1]{BFOR2017} for instance).
\end{proof}

\subsection{The case of $n=0$}

Here we calculate $\langle \phi_{0, \nu}(\cdot, s), f\rangle$ at $s = 3/4$ for a normalized Hecke eigenform $f$.
To this end, we decompose $\phi_{0,\nu}(\tau,s)$ in terms of the Selberg--Poincar\'e series.

\begin{proposition}\label{prop:phi-nu-Poincare} We have that
	\[
		\phi_{0,\nu}(\tau, s) = 4^{-s+3/4} \sum_{0 \le l \le \nu} \frac{(s-3/4)^{\un{l}}}{(4\pi)^l} \binom{2\nu+1}{2l+1} \sum_{r \in \Z} r^{2\nu-2l} P_{2\nu+2, r^2} \left(\tau, s-\frac{3}{4}-l\right).
	\]
\end{proposition}

\begin{proof}
	By applying \eqref{eq:Lk-slash}, \cref{lem:Theta-Poincare}, and \cref{prop:LLL-p},
	\begin{align*}
		\phi_{0,\nu} (\tau, s) &= \lim_{z \to 0} \sum_{\gamma \in \Gamma_\infty \backslash \Gamma} \sum_{\substack{r \in \Z \\ r \equiv 0\ (2)}} \left(L_{2\nu} \circ \cdots \circ L_2 \left(\frac{v}{4}\right)^{s-3/4} q^{r^2/4} \zeta^r \right) \bigg|_{2\nu+2,1} \gamma\\
		&= 4^{-s+3/4} \sum_{\gamma \in \Gamma_\infty \backslash \Gamma} \sum_{r \in \Z} p_{2\nu+2}(2r,D) v^{s-3/4} q^{r^2} \bigg|_{2\nu+2} \gamma.
	\end{align*}
	The summand is calculated as
	\begin{align*}
		p_{2\nu+2}(2r,D) v^{s-3/4} q^{r^2} &= \sum_{0 \le j \le \nu} (-1)^j \binom{2\nu-j}{j} (2r)^{2\nu-2j} D^j \left(v^{s-3/4} q^{r^2} \right).
	\end{align*}
	Then the claim follows from the Leibniz rule, where $D v^{s-3/4} = \frac{-1}{4\pi} \frac{\mathrm{d}}{\mathrm{d}v} v^{s-3/4}$, and the fact that
	\[
		\sum_{l \le j \le \nu} (-1)^j 2^{2(\nu-j)} \binom{2\nu-j}{j} \binom{j}{l} = (-1)^l \binom{2\nu+1}{2l+1}.
	\]
\end{proof}

The next result provides a formula for the Petersson norm of a cusp form $f$.

\begin{theorem}\label{thm:phinu-f-Pet}
	For a normalized Hecke eigenform $f \in S_{2\nu+2}$, we have
	\[
		\lim_{s \to 3/4} \langle \phi_{0,\nu} (\cdot, s), f \rangle = 24 \|f\|^2.
	\]
\end{theorem}

\begin{proof}
	First, we note that the Fourier coefficients of a normalized Hecke eigenform are real. By \cref{lem:Poincare-Petersson} and \cref{prop:phi-nu-Poincare}, we find that
	\begin{align*}
		\langle \phi_{0,\nu}(\cdot, s), f \rangle &= 4^{-s+3/4} \sum_{0 \le l \le \nu} \frac{(s-3/4)^{\un{l}}}{(4\pi)^l} \frac{\Gamma(2\nu+s+1/4-l)}{(4\pi)^{2\nu+s+1/4-l}} \binom{2\nu+1}{2l+1} \cdot 2 \sum_{r =1}^\infty \frac{c_f(r^2)}{r^{2s+2\nu+1/2}}.
	\end{align*}
	As in~\cite[Lemma 11.12.6]{CohenStromberg2017}, let
	\[
		B(f, s) \coloneqq \sum_{n=1}^\infty \frac{c_f(n^2)}{n^s} = \frac{1}{\zeta(2(s-2\nu-1))} L(\mathrm{Sym}^2(f), s)
	\]
	for $f \in S_{2\nu+2}$. Then it is known that $B(f,s)$ admits the meromorphic continuation to the whole $\C$-plane, and $L(\mathrm{Sym}^2(f), s)$ has no poles, (see~\cite[Remark 11.12.8]{CohenStromberg2017}). In particular, $B(f, 2s+2\nu+1/2)$ has no pole at $s = 3/4$. Therefore, by \cite[Corollary 11.12.7]{CohenStromberg2017}, we have
	\begin{align*}
		\lim_{s \to 3/4} \langle \phi_{0,\nu}(\cdot, s), f\rangle &= \frac{\Gamma(2\nu+1)}{(4\pi)^{2\nu+1}} (2\nu+1) 2B(f, 2\nu+2)\\
		&= \frac{2(2\nu+1)!}{(4\pi)^{2\nu+1}} \frac{6}{\pi^2} \frac{\pi}{2} \frac{(4\pi)^{2\nu+2}}{(2\nu+1)!} \langle f, f \rangle\\
		&=24 \|f\|^2.	
	\end{align*}
\end{proof}

\subsection{The cases of $n > 0$}

We turn to the case of positive  $n.$ Again, we first decompose $\phi_{n,\nu}(\tau,s).$

\begin{proposition}\label{prop: phi-nnu-calc}
	For $n > 0$, we have
	\begin{align*}
		\phi_{n, \nu}(\tau, s) &= \frac{1}{\Gamma(2s)} \sum_{\substack{r \in \Z \\ r \equiv n\ (2)}} \sum_{\substack{i_1, i_2 \ge 0 \\ i_1 + i_2 \le \nu}} \frac{(-1)^{i_1}}{i_1! i_2!} \left(\frac{n^2}{4} \right)^{i_1+i_2} Q_{\nu, i_1+i_2} (n,r) \frac{(s-3/4)^{\un{i_1}} (s-3/4)^{\ov{i_2}}}{(2s)^{\ov{i_2}}}\widetilde{P}_{n,r}^{i_1, i_2} (\tau, s),
	\end{align*}
	where we let
	\begin{align*}
		Q_{\nu, i}(n,r) &\coloneqq \sum_{i \le j \le \nu} (-1)^j \binom{2\nu-j}{j} r^{2\nu-2j} \frac{j!}{(j-i)!} \left(\frac{r^2-n^2}{4}\right)^{j-i},\\
		\widetilde{P}_{n,r}^{i_1, i_2} (\tau, s) &\coloneqq \sum_{\gamma \in \Gamma_\infty \backslash \Gamma} (\pi n^2v)^{-3/4-i_1-i_2/2} M_{-3/4+i_2/2, s-1/2+i_2/2}(\pi n^2 v) \e\left(\frac{r^2 - n^2}{4} u\right) e^{-\frac{\pi r^2 v}{2}} \bigg|_{2\nu+2} \gamma.
	\end{align*}
\end{proposition}

\begin{proof}
	Arguing as above, by applying \eqref{eq:Lk-slash}, \cref{lem:Theta-Poincare}, and \cref{prop:LLL-p}, we obtain
	\begin{align*}
		\phi_{n,\nu}(\tau, s) &= \lim_{z \to 0} \sum_{\gamma \in \Gamma_\infty \backslash \Gamma} \sum_{\substack{r \in \Z \\ r \equiv n\ (2)}} \bigg(L_{2\nu} \circ \cdots \circ L_2 \calM_{3/2, -n^2}\left(\frac{v}{4}, s\right) \e \left(\frac{-n^2u}{4}\right) q^{r^2/4} \zeta^r \bigg) \bigg|_{2\nu+2, 1} \gamma\\
		&= \sum_{\gamma \in \Gamma_\infty \backslash \Gamma} \sum_{\substack{r \in \Z \\ r \equiv n\ (2)}} p_{2\nu+2}(r, D) \calM_{3/2, -n^2} \left(\frac{v}{4}, s\right) \e\left(\frac{-n^2u}{4}\right) q^{r^2/4} \bigg|_{2\nu+2} \gamma.
	\end{align*}
	The summand is calculated as
	\begin{align*}
		&p_{2\nu+2}(r, D) \calM_{3/2, -n^2} \left(\frac{v}{4}, s\right) \e\left(\frac{-n^2u}{4}\right) q^{r^2/4}\\
		&= \frac{1}{\Gamma(2s)} \sum_{0 \le j \le \nu} (-1)^j \binom{2\nu-j}{j} r^{2\nu-2j} D^j \left[(\pi n^2 v)^{s-3/4} \cdot (\pi n^2 v)^{-s} M_{-3/4, s-1/2}(\pi n^2 v) e^{-\frac{\pi n^2v}{2}} \cdot q^{\frac{r^2-n^2}{4}} \right]\\
		&= \frac{1}{\Gamma(2s)} \sum_{0 \le j \le \nu} (-1)^j \binom{2\nu-j}{j} r^{2\nu-2j} \\
		&\qquad \times \sum_{\substack{i_1, i_2, i_3 \ge 0\\ i_1 + i_2 + i_3 = j}} \frac{j!}{i_1! i_2! i_3!} D^{i_1}(\pi n^2 v)^{s-3/4} \cdot D^{i_2} \left[ (\pi n^2 v)^{-s} M_{-3/4, s-1/2}(\pi n^2 v) e^{-\frac{\pi n^2v}{2}} \right] \cdot D^{i_3} q^{\frac{r^2-n^2}{4}}.
	\end{align*}
	Similar to the case of $n=0$, direct calculation utilizing $D f(v) = \frac{-1}{4\pi} \frac{\mathrm{d}}{\mathrm{d}v} f(v)$ yields
	\begin{align*}
		D^{i_1}(\pi n^2 v)^{s-3/4} 
		&= \left(-\frac{n^2}{4} \right)^{i_1} (s-3/4)^{\un{i_1}} (\pi n^2 v)^{s-3/4-i_1},\\
		D^{i_3} q^{\frac{r^2-n^2}{4}} &= \left(\frac{r^2-n^2}{4}\right)^{i_3} q^{\frac{r^2-n^2}{4}}.
	\end{align*}
	For the second term, by \cref{lem:M-Whit-der}, we find that
	\begin{align*}
		&D^{i_2} \left[ (\pi n^2 v)^{-s} M_{-3/4, s-1/2}(\pi n^2 v) e^{-\frac{\pi n^2v}{2}} \right]\\
		&= \left(\frac{n^2}{4}\right)^{i_2} \frac{(s-3/4)^{\ov{i_2}}}{(2s)^{\ov{i_2}}} e^{-\frac{\pi n^2v}{2}} (\pi n^2 v)^{-s-i_2/2} M_{-3/4+i_2/2, s-1/2+i_2/2}(\pi n^2 v).
	\end{align*}
	The claim follows by combining these results.
\end{proof}

We split the sum defining $\phi_{n, \nu}(\tau, s)$ into $\phi_{n, \nu}^+(\tau, s)$ and $\phi_{n, \nu}^-(\tau, s),$ based on the inequalities $r^2 > n^2$ or $r^2 \le n^2.$  We consider them as $s \to 3/4$. By \eqref{eq:W-asymp}, the summand of the Poincar\'{e} series $\widetilde{P}_{n,r}^{i_1, i_2} (\tau, s)$ satisfies
\[
	v^{-3/4-i_1-i_2/2} M_{-3/4+i_2/2, s-1/2+i_2/2}(\pi n^2 v) \e\left(\frac{r^2 - n^2}{4} u\right) e^{-\frac{\pi r^2 v}{2}} = O(v^{\Re(s)-3/4-i_1})
\]
as $v \to 0$. Therefore, for $\Re(s) > -\nu +i_1 + 3/4$, the Poincar\'{e} series is holomorphic (in $s$). In particular, $\widetilde{P}_{n,r}^{i_1, i_2} (\tau, s)$ is holomorphic at $s = 3/4$ for $0 \le i_1 < \nu$. Regarding the case of $i_1 = \nu$, by a similar argument as in \cref{section:selberg-Poincare}, that is, by comparing it with the Selberg--Poincar\'{e} series or the Maass--Poincar\'{e} series, we see that it is also holomorphic at $s=3/4$.
Therefore, we have
\begin{align*}
	\lim_{s \to 3/4} \phi_{n, \nu}^-(\tau, s) &= \frac{1}{\Gamma(3/2)} \sum_{\substack{r \in \Z \\ r \equiv n\ (2)\\r^2 \le n^2}} Q_{\nu,0} (n,r) \widetilde{P}_{n,r}^{0,0}(\tau, 3/4).
\end{align*}
Since $Q_{\nu,0}(n,r) = p_{2\nu+2}(r, (r^2-n^2)/4)$ and $\widetilde{P}_{n,r}^{0,0}(\tau, 3/4) = P_{2\nu+2, \frac{r^2-n^2}{4}}(\tau)$, by \eqref{eq:pkr-exp-eq}, we have
\begin{align}\label{eq:minus-part}
	\lim_{s \to 3/4} \phi_{n, \nu}^-(\tau, s) = \frac{4}{n \sqrt{\pi}} \sum_{0 < r \le n} r^{2\nu+1} P_{2\nu+2, -r(n-r)}(\tau).
\end{align}

As a counterpart to \cref{thm:phinu-f-Pet}, the Petersson inner product of $\phi_{n, \nu}^+(\tau, s)$ with a Hecke eigenform is expressed in terms of the symmetrized shifted convolution $L$-functions.

\begin{theorem}\label{thm:phi+-f-Peter}
	For a normalized Hecke eigenform $f \in S_{2\nu+2}$, we have
	\begin{align*}
		\lim_{s \to 3/4} \langle \phi_{n,\nu}^+(\cdot, s), f \rangle = \frac{4}{n \sqrt{\pi}} \frac{\Gamma (2\nu + 1)}{(4\pi)^{2\nu + 1}} \sum_{d \mid n} \mu(d) \widehat{L}(f, n/d; 2\nu+1).
	\end{align*}
\end{theorem}

\begin{proof}
	By \cref{prop: phi-nnu-calc}, we have
	\begin{align*}
		&\langle \phi_{n, \nu}^+(\cdot, s), f \rangle\\
		&= \frac{1}{\Gamma(2s)} \sum_{\substack{r \in \Z \\ r \equiv n\ (2) \\ r^2 > n^2}} \sum_{\substack{i_1, i_2 \ge 0 \\ i_1 + i_2 \le \nu}} \frac{(-1)^{i_1}}{i_1! i_2!} \left(\frac{n^2}{4} \right)^{i_1+i_2} Q_{\nu, i_1+i_2} (n,r) \frac{(s-3/4)^{\un{i_1}} (s-3/4)^{\ov{i_2}}}{(2s)^{\ov{i_2}}} \langle \widetilde{P}_{n,r}^{i_1, i_2} (\cdot, s), f \rangle.
	\end{align*}
	The unfolding argument, combined with \cref{lem:M-Whit-int}, gives
	\begin{align*}
		&(\pi n^2)^{\frac{3}{4}+i_1+\frac{i_2}{2}} \langle \widetilde{P}_{n,r}^{i_1, i_2} (\cdot, s), f \rangle\\
		&= \sum_{m=1}^\infty c_f(m) \int_0^\infty \int_0^1 v^{2\nu-\frac{3}{4}-i_1-\frac{i_2}{2}} M_{-\frac{3}{4}+\frac{i_2}{2}, s-\frac{1}{2}+\frac{i_2}{2}}(\pi n^2 v) \e\left(\left(\frac{r^2 - n^2}{4} -m\right) u\right) e^{-2\pi \left(\frac{r^2}{4}+m\right)v} \mathrm{d}u \mathrm{d}v\\
		&= c_f \left(\frac{r^2-n^2}{4}\right) \int_0^\infty e^{-\pi \left(r^2-\frac{n^2}{2}\right)v} v^{2\nu-\frac{3}{4}-i_1-\frac{i_2}{2}} M_{-\frac{3}{4}+\frac{i_2}{2}, s-\frac{1}{2}+\frac{i_2}{2}}(\pi n^2 v) \mathrm{d}v\\
		&= c_f \left(\frac{r^2-n^2}{4}\right) \frac{(\pi n^2)^{s + \frac{i_2}{2}} \Gamma \left(s + 2\nu + \frac{1}{4} - i_1 \right)}{(\pi r^2)^{s + 2\nu + \frac{1}{4} - i_1}} \cdot {}_2 F_1 \left(s+\frac{3}{4}, s + 2\nu + \frac{1}{4} - i_1; 2s+i_2; \frac{n^2}{r^2}\right).
	\end{align*}	
	By changing variables $r = 2m+n$ for $r > n$ and $r = -2m-n$ for $r < -n$, we have
	\begin{align*}
		&\langle \phi_{n,\nu}^+(\cdot, s), f\rangle\\
		&= \frac{2}{\Gamma(2s)} \sum_{\substack{i_1, i_2 \ge 0 \\ i_1 + i_2 \le \nu}} \frac{(-1)^{i_1}}{i_1! i_2!} \left(\frac{n^2}{4} \right)^{i_1+i_2} \frac{(s-3/4)^{\un{i_1}} (s-3/4)^{\ov{i_2}}}{(2s)^{\ov{i_2}}} (\pi n^2)^{s -\frac{3}{4}-i_1} \Gamma \left(s + 2\nu + \frac{1}{4} - i_1 \right)\\
		&\times \sum_{m=1}^\infty \frac{Q_{\nu, i_1+i_2} (n,2m+n) c_f (m(m+n))}{(\pi (2m+n)^2)^{s + 2\nu + \frac{1}{4} - i_1}} {}_2 F_1 \left(s+\frac{3}{4}, s + 2\nu + \frac{1}{4} - i_1; 2s+i_2; \frac{n^2}{(2m+n)^2}\right),
	\end{align*}
	where we note that $Q_{\nu,i}(n,-r) = Q_{\nu,i}(n,r)$ holds. For a normalized Hecke eigenform $f \in S_{2\nu+2}$, since
	\[
		c_f(m(m+n)) = \sum_{d \mid (m, m+n)} \mu(d) d^{2\nu+1} c_f\left(\frac{m}{d}\right) c_f\left(\frac{m+n}{d}\right),
	\]
	the last sum becomes
	\begin{align*}
		&\sum_{d \mid n} \mu(d) d^{2\nu+1} \sum_{m=1}^\infty \frac{Q_{\nu, i_1+i_2} (n,2dm+n)  c_f(m) c_f(m+n/d)}{(\pi (2dm+n)^2)^{s + 2\nu + \frac{1}{4} - i_1}}\\
		&\quad \times {}_2 F_1 \left(s+\frac{3}{4}, s + 2\nu + \frac{1}{4} - i_1; 2s+i_2; \frac{n^2}{(2dm+n)^2}\right).
	\end{align*}
	Then, this Dirichlet series is holomorphic at $s = 3/4$. Indeed, since $Q_{\nu, i_1+i_2} (n,2dm+n)$ has degree $2(\nu-i_1-i_2)$ in $m$, it suffices to show that
	\[
		\sum_{m=1}^\infty \frac{c_f(m) c_f(m+n/d)}{m^{2(s+\nu+1/4+i_2)}}
	\]
	(conditionally) converges at $s = 3/4$. This can be seen by partial summation using the estimate
	\[
		\sum_{1 \le m \le x} c_f(m) c_f(m+n/d) \ll x^{2\nu+2-\delta},
	\]
	with some $\delta > 0$, (see~\cite[Corollary 1.4]{Blomer2004}). Therefore, all terms corresponding to non-zero $(i_1, i_2)$ vanish as $s \to 3/4$, and we obtain
	\begin{align*}
		\lim_{s \to 3/4} \langle \phi_{n,\nu}^+(\cdot, s), f\rangle &= \frac{4}{\sqrt{\pi}} \Gamma(2\nu+1) \sum_{d \mid n} \mu(d) d^{2\nu+1} \\
		&\quad \times \sum_{m=1}^\infty \frac{Q_{\nu, 0} (n,2dm+n)  c_f(m) c_f(m+n/d)}{(\pi (2dm+n)^2)^{2\nu + 1}} {}_2 F_1 \left(\frac{3}{2}, 2\nu + 1; \frac{3}{2}; \frac{n^2}{(2dm+n)^2}\right).
	\end{align*}
	Since we have
	\[
		\frac{1}{r^{2(2\nu+1)}} {}_2 F_1 \left(\frac{3}{2}, 2\nu + 1; \frac{3}{2}; \frac{n^2}{r^2}\right) = \frac{1}{(r^2-n^2)^{2\nu+1}}
	\]
	and $Q_{\nu,0}(n,r) = p_{2\nu+2}(r, (r^2-n^2)/4)$ with \eqref{eq:pkr-exp-eq} again, the proof is complete as
	\begin{align*}
		\lim_{s \to 3/4} \langle \phi_{n,\nu}^+(\cdot, s), f\rangle &= \frac{4}{n \sqrt{\pi}} \frac{\Gamma(2\nu+1)}{(4\pi)^{2\nu + 1}} \sum_{d \mid n} \mu(d) \sum_{m=1}^\infty c_f(m) c_f(m+n/d) \left(\frac{1}{m^{2\nu+1}} - \frac{1}{(m+n/d)^{2\nu+1}}\right).
	\end{align*}
	\end{proof}

\subsection{Proof of \cref{Main-Theorem}}

We apply the results from the two previous subsections.
For $m > 0$, let
\begin{align*}
	\phi(\tau, s) \coloneqq \frac{\sqrt{\pi}}{4} \sum_{n \mid m} n \phi_{n,\nu}(\tau, s) - \sigma_1(m) \phi_{0,\nu}(\tau, s).
\end{align*}
As stated in \eqref{eq:RC-heat}, we have
\[
	\lim_{s \to 3/4} \phi(\tau, s) = \calG_{m,\nu}(\tau).
\]
On the other hand, from \eqref{eq:minus-part}, the minus part 
\begin{align*}
	\lim_{s \to 3/4} \phi^-(\tau, s) \coloneqq \lim_{s \to 3/4} \frac{\sqrt{\pi}}{4} \sum_{n \mid m} n \phi_{n, \nu}^-(\tau, s) = \sum_{n \mid m} \sum_{0 < r \le n} r^{2\nu+1} P_{2\nu+2, -r(n-r)}(\tau).
\end{align*}
For the plus part, by \cref{thm:phinu-f-Pet} and \cref{thm:phi+-f-Peter} and the M\"{o}bius inversion formula, we have
\begin{align*}
	\lim_{s \to 3/4} \langle \phi^+(\cdot, s), f \rangle
	 &= \sum_{n \mid m} \frac{\Gamma (2\nu + 1)}{(4\pi)^{2\nu + 1}} \sum_{d \mid n} \mu(d) \widehat{L}(f, n/d; 2\nu+1) - 24\sigma_1(m) \|f\|^2\\
	&= \frac{\Gamma (2\nu + 1)}{(4\pi)^{2\nu + 1}} \widehat{L}(f, m; 2\nu+1) - 24 \sigma_1(m) \|f\|^2.
\end{align*}
Combining these facts, we are pleased to obtain the conclusion of the theorem
\begin{align*}
	\lim_{s \to 3/4} \phi(\tau, s) &= \sum_{n \mid m} \sum_{0 < r \le n} r^{2\nu+1} P_{2\nu+2, -r(n-r)}(\tau)
 - \sum_{j=1}^{d_{2\nu+2}} \left(24 \sigma_1(m) - \frac{\Gamma (2\nu + 1)}{(4\pi)^{2\nu + 1}} \frac{\widehat{L}(f, m; 2\nu+1)}{\|f_j\|^2} \right) f_j.
\end{align*}

\bibliographystyle{amsalpha}
\bibliography{References}

\end{document}